\documentclass[11pt,letterpaper]{scrarticle}
\usepackage[utf8]{inputenc}
\usepackage[T1]{fontenc}
\usepackage{amsmath}
\usepackage{amsfonts}
\usepackage{amssymb}
\usepackage{amsthm}
\usepackage{mathtools}
\usepackage{euscript}
\usepackage[dvipsnames]{xcolor}

\usepackage{hyperref}
\usepackage{subcaption}
\usepackage[inner=1in,outer=1in]{geometry}
\usepackage{tikz}
\usepackage{tikz-cd}
\usetikzlibrary{positioning, calc}
\usetikzlibrary{3d}
\usetikzlibrary{decorations.markings}
\tikzstyle{vertex}=[inner sep=0pt]

\theoremstyle{plain}
\newtheorem{thm}{Theorem}[section]
\newtheorem*{thm*}{Theorem}
\newtheorem{prop}[thm]{Proposition}

\newtheorem{cor}[thm]{Corollary}

\theoremstyle{definition}
\newtheorem{defn}[thm]{Definition}
\newtheorem{rmk}[thm]{Remark}

\theoremstyle{remark}
\newtheorem{example}[thm]{Example}

\numberwithin{equation}{section}

\def\on{\operatorname}
\def\sf{\mathsf}

\def\Fun{{\on{Fun}}}

\def\QConv{{\on{QConv}}}
\def\Comm{{\scr{C}\!\on{omm}}}

\def\Cat{{\sf{Cat}}}
\def\Set{{\sf{Set}}}

\def\Assoc{{\scr{A}\!\on{ssoc}}}

\def\Fin{{\sf{Fin}}}

\def\id{{\on{id}}}

\def\un{{\underline{n}}}
\def\um{{\underline{m}}}
\def\ul{{\underline{l}}}

\def\bb{\mathbb}
\def\CC{{\bb{C}}}

\def\scr{\EuScript}
\def\I{{\sf{I}}}
\def\J{{\sf{J}}}
\def\C{{\sf{C}}}
\def\D{{\sf{D}}}
\def\P{{\scr{P}}}
\def\O{{\scr{O}}}
\def\Q{{\scr{Q}}}

\def\DFib{{\sf{DFib}}}
\def\ISet{{\sf{ISet}}}

\def\un{{\underline{n}}}
\def\um{{\underline{m}}}

\definecolor{bluegray}{rgb}{0.4, 0.6, 0.8}
\definecolor{turquoise}{rgb}{0.2, 0.7, 0.6}

\usepackage[backend=biber]{biblatex}
\title{An $\scr{O}$-monoidal Grothendieck construction}
\author{Redi Haderi and Walker H. Stern}

\begin{document}
	
	\maketitle
	
	\begin{abstract}
		Given an operad $\O$, we define a notion of weak $\O$-monoids --- which we term $\O$-pseudomonoids -- in a 2-category. In the special case with the 2-category in question is the 2-category $\Cat$ of categories, this yields a notion of $\O$-monoidal category, which in the case of the associative and commutative operads retrieves unbiased notions of monoidal and symmetric monoidal categories, respectively. We carefully unpack the definition of $\O$-monoids in the 2-categories of discrete fibrations and of category-indexed sets. Using the classical Grothendieck construction, we thereby obtain an $\O$-monoidal Grothendieck construction relating lax $\O$-monoidal functors into $\Set$ to strict $\O$-monoidal functors which are also discrete fibrations.  
	\end{abstract}
	
	\section{Introduction}
	
	Category theory provides numerous ways of characterizing and studying notions of `algebraic structure', whether through diagrammatic presentations like those of monoid and group objects in categories, or through more general theoretical frameworks like monads, Lawvere theories, or operads (see, e.g., \cite[\S 5.4]{Benabou}, \cite{LawvereTheory}, and \cite{MSS,Leinster}, respectively). Regardless of the precise structure being characterized, or the categorical framework used to study it, however, the same problem arises when trying to push one's way up the ladder of categorification: determining how best to characterize \emph{coherent} algebraic structures so that classical statements in the 1-categorical context generalize naturally to higher categories.
	
	The present paper concerns itself with just such a problem: that transporting algebraic structures coherently across the Grothendieck construction\footnote{While the underlying construction on objects which we use is the original Grothendieck construction of \cite{Grothendieck}, we formulate the equivalence globally, without a fixed base.} in a sensible way. In a sense, most of this paper involves setting the stage for a theorem which is then immediate: painstakingly defining, describing, and working with appropriately defined coherent algebras over operads until we may simply use the classical Grothendieck construction to state our desired equivalence. 
	
	Before exposing our main result in full generality or delving too deeply into the constructions which comprise this paper, let us first dwell on the 1-categorical case in which our construction becomes wholly transparent and mostly trivial. If we consider sets as discrete categories, restricting the  Grothendieck construction to functors $X\to \Set$ yields an equivalence between, on the one hand, the arrow category of sets and, on the other hand, the category of \emph{indexed sets}. The idea, ubiquitous in modern mathematics, is that a set $\{X_y\}_{y\in Y}$ of sets is equivalent to a map of sets 
	\[
	\begin{tikzcd}
		\displaystyle \coprod_{y\in Y} X_y \arrow[r] & Y 
	\end{tikzcd}
	\]
	which sends $x\in X_y$ to $y$. The categorical component of this equivalence comes into play by noting that given indexed sets $\{X_y\}_{y\in Y}$ and $\{A_b\}_{b\in B}$, a pair consisting of a map 
	\[
	\begin{tikzcd}
		f:&[-3em] B\arrow[r] & Y 
	\end{tikzcd}
	\]
	and an indexed set 
	\[
	\{\!\begin{tikzcd}
		g_b:&[-3em] A_b\arrow[r] & X_{f(b)} 
	\end{tikzcd}\!\!\}_{b\in B}
	\]
	can be uniquely identified with a commutative diagram 
	\[
	\begin{tikzcd}
		\coprod_{b\in B} A_b\arrow[r,"g"]\arrow[d]&\coprod_{y\in Y} X_y\arrow[d]\\
		B \arrow[r,"f"] & Y 
	\end{tikzcd}
	\]
	of sets. 
	
	If one then studies, e.g., associative algebras on both sides, one finds that an associative algebra in the arrow category is a pair of monoids and a homomorphism between them, and that an associative algebra in the category of indexed sets consists of a  monoid $(M,\cdot,1)$, a set $\{X_m\}_{m\in M}$ of sets indexed by $M$, and maps 
	\[
	\begin{tikzcd}
		X_{m}\times X_n \arrow[r]& X_{m\cdot n}
	\end{tikzcd}
	\]
	which satisfy some associativity and unitality conditions. If one instead studies algebras over some operad $\scr{O}$, one gets effectively the same kind of structure, but now satisfying whatever algebraic conditions are defined by $\scr{O}$.
	
	In this discrete setting, the correspondence between these two structures is transparent. However the nature of the latter ``monoid-indexed'' sets only becomes clear when returning to the viewpoint of 1-categories. One can view a monoid as a discrete strict monoidal category. From this viewpoint, the maps $X_m\times X_n\to X_{m\times n}$ form the structure maps of a lax monoidal functor from this monoidal category to the category of sets with the Cartesian product. 
	
	The purpose of this paper is to develop the necessary definitions and lemmata extend this correspondence to discrete fibrations of categories on the one hand and functors
	\[
	\begin{tikzcd}
		I\arrow[r] & \Set 
	\end{tikzcd}
	\]
	on the other. More precisely, for an operad $\O$, we aim to prove a statement of the form 
	\[
	\left\lbrace \substack{
		\text{discrete fibrations of}\\
		\text{weak } \O\text{-algebras}
	}\right\rbrace \simeq \left\lbrace \substack{
	\text{Lax $\O$-monoidal functors }\\
   I\longrightarrow \Set}
\right\rbrace
	\]
	which generalizes the discrete case above. One such generalization (in the cases of associative, braided, and symmetric monoidal categories) is already present in the literature in \cite{MoellerVasilakopoulou}, and a related equivalence can be found in \cite{Shulman}. In these cases, the equivalence presented is more general than that given here, since it works with pseudofunctors into $\Cat$ and Grothendieck fibrations, rather than functors into $\Set$ and discrete fibrations.
	
	While we believe that our results generalize to the setting of Grothendieck fibrations, we choose to work with discrete fibrations for a number of reasons. Primary among them is that our definitions and results rely on quite involved 2-categorical coherence conditions, even in the discrete case. Attempting to work with these \emph{in addition to} the coherence conditions for pseudofunctors inherent in the study of Grothendieck fibrations would substantially increase the length and complexity of the proofs in this work, and would make the arguments more opaque. The other major reason for working with discrete fibrations is that our primary intended application --- to the study of convexity --- requires only the discrete case. The details of this application are exposed in \cite{HOSconvex}. 
	
	\subsection{$\O$-pseudomonoids}
	
	We now turn to explaining the key concepts necessary to enable the generalization above. There are two categories key to the present work. The first of these is the category $\DFib$ of \emph{discrete fibrations}, whose objects are discrete fibrations of categories $\pi:\C_1\to \C_0$ and whose morphisms are  (strictly) commutative diagrams
	\[
	\begin{tikzcd}
		\C_1 \arrow[r]\arrow[d,"\pi_{\C}"'] & \D_1 \arrow[d,"\pi_{\D}"]\\
		\C_0 \arrow[r] & \D_0 
	\end{tikzcd}
	\]
	of functors. The second of these is the category $\ISet$ of \emph{indexed sets}\footnote{We here borrow terminology and notation from \cite{MoellerVasilakopoulou}, as it is particularly well suited to this subject.}, whose objects are functors 
	\[
	\begin{tikzcd}
		\C\arrow[r] & \Set 
	\end{tikzcd}
	\]
	and whose morphisms are diagrams 
	\[
	\begin{tikzcd}
		\C\arrow[dr,""'{name=U}]\arrow[dd] & \\
		& \Set \\
		\D\arrow[ur,] & \arrow[from=U,to=3-1,Rightarrow,"\mu",shorten >= 0.5em]
	\end{tikzcd}	
	\]
	where $\mu$ is a natural transformation filling the triangle. 
	
	However, neither of these are merely 1-categories, but rather 2-categories. Both have natural 2-morphisms, consisting of natural transformations between the functors comprising the 1-morphisms, subject to some compatibility conditions. The (discrete) Grothendieck construction can then be given the form of a 2- equivalence 
	\[
	\begin{tikzcd}
		\displaystyle \int: &[-3em] \ISet \arrow[r,"\simeq"] & \DFib. 
	\end{tikzcd}
	\] 
	While this `global' formulation of the Grothendieck construction is folklore, the only reference in the literature we know for it is \cite[Thm. 2.3]{MoellerVasilakopoulou}. Our main theorem aims to take an symmetric, one-colored operad $\scr{O}$ (in $\Set$), and show that this construction generalizes to the case where all the categories involved are $\O$-algebras in $\Cat$, and all the functors involved preserve the $\O$-algebra structures. 
	
	However, the 2-categorical nature of the characters of our story complicates studying algebras over operads. For instance, naturally occurring monoidal structures tend to be \emph{weakly} associative and unital, rather than strictly so. As such, to capture and generalize these examples, we need a notion of \emph{weak} $\O$-algebra in a 2-category. 
	
	Our definition of $\O$-\emph{pseudomonoid} (see Definition \ref{defn:pseudomonoid}) is one such notion appropriate for our setting. For other, related definitions, see \cite{CornerGurski} and \cite{GMMO}. Loosely speaking, it requires that the associativity of the operad action hold up to specified 2-isomorphisms, and that these 2-isomorphisms satisfy coherence conditions. We use the term \emph{pseudomonoid} rather than \emph{pseudoalgebra} because we phrase all of our definitions in terms of products, rather than in the generality of symmetric monoidal 2-categories.   
	
	In the special case of $\scr{O}$-pseudomonoids in the 2-category $\Cat$ of categories, we use the term $\O$-\emph{monoidal categories}. This term is motivated by the examples of the associative and commutative operads $\Assoc$ and $\Comm$, over which $\O$-monoidal categories are monoidal categories, and symmetric monoidal categories, respectively. However, the term $\O$-monoidal category appears elsewhere in the literature, in particular, \cite[Definition 2.1.2.13]{HA} defines a notion of \emph{$\O$-monoidal $\infty$-category}. The definition requires a fair amount of $\infty$-categorical technology, but ends up being equivalent to the notion of a coherent $\O$-monoid in $\Cat_\infty$ (see, \cite[Example 2.4.2.4]{HA}). While we will not prove that our definition represents a specialization of that of \cite{HA}, this observation justifies the use of the same terminology for both definitions. 
	
	The functoriality of this definition in morphisms of operads, together with the fact that $\Comm$ is the terminal operad shows that any symmetric monoidal category carries an $\O$-monoidal structure induced by the symmetric monoidal structure. Of particular import, this means we can consider $(\Set,\times)$ as an $\O$-monoidal category for any $\O$.
	
	Much of the paper is devoted to proving two equivalences of 2-categories, one for each side of the Grothendieck construction. Each proof is, effectively, an unwinding of definitions. The first shows that $\O$-pseudomonoids in $\DFib$ are equivalent to discrete fibrations $\pi:\C_1\to \C_0$ such that $\C_1$ and $\C_0$ are $\O$-monoidal categories, and $\pi$ is a strict $\O$-monoidal fibrations. Symbolically, we write
	\[
	\O\sf{Mon}(\DFib)\simeq \O\sf{Fib}.
	\]
	The second shows an $\O$-pseudomonoid in $\ISet$ consists of an $\O$-monoidal category $(\C,\otimes)$ and a lax $\O$-monoidal functor 
	\[
	\begin{tikzcd}
		(\C,\otimes)\arrow[r] & (\Set,\times). 
	\end{tikzcd}
	\]
	In symbols, we write 
	\[
	\O\sf{Mon}(\ISet)\simeq \ISet^{\O,\on{lax}}.
	\]
	
	\subsection{The Main theorem}
	
	With the main definitions established and unwound, the main theorem of this paper becomes relatively simple to state. 
	
	\begin{thm*}
		For any operad $\scr{O}$, the classical Grothendieck construction $\int:\ISet\to \DFib$ induces an equivalence of 2-categories 
		\[
		\begin{tikzcd}
			{\displaystyle \int^{\scr{O}}}: &[-3em] \ISet^{\scr{O},\on{lax}}\arrow[r] & \scr{O}\sf{Fib}. 
		\end{tikzcd}
		\]
	\end{thm*}

	It is worth pointing out that we expect a number of generalizations of this theorem to hold. In particular, (1) we expect an analogous result to hold for indexed categories and Grothendieck fibrations, and (2) we expect a similar equivalence of $\infty$-categories for any $\infty$-operad $\O$, induced by the Grothendieck-Lurie construction \cite[Thm. 3.2.0.1]{HTT}.
	
	\subsection{Structure of the paper}
	
	The paper is laid out as follows. In section \ref{sec:discGroth}, we review the classical ($\Set$-valued) Grothendieck construction, and provide a proof. We then define $\O$-pseudomonoids, their morphisms and 2-morphisms, and prove some basic properties in section \ref{sec:opsmon}. Section \ref{sec:opsCat} explains the way in which $\Assoc$- and $\Comm$-monoidal categories correspond to monoidal and symmetric monoidal categories, respectively. In sections \ref{sec:opsDFib} and \ref{sec:opsISet} we unravel the definitions of $\O$-pseudomonoids in $\DFib$ and $\ISet$, and relate them to $\O$-monoidal categories. The short section \ref{sec:OGroth} then assembles the results of the previous sections to state and prove the main theorem.
	
	\subsection*{Acknowledgements}
	This work is supported by the US Air Force Office of Scientific Research under award number FA9550-21-1-0002.

	\section{The discrete Grothendieck construction}\label{sec:discGroth}
	
	We begin by reviewing the discrete Grothendieck construction --- effectively the restriction to discrete fibrations of \cite[Thm. 2.3]{MoellerVasilakopoulou} --- in detail. For another discussion of discrete fibrations and the concomitant Grothendieck construction, see \cite[\S 2]{RLfibrations}.
	
	\begin{defn}
		A \emph{discrete (coCartesian) fibration} is a functor $p:\C_1\to \C_0$ such that, for every $c\in \C_1$ and every morphism $f:p(c)\to d$ in $\C_0$, there is a \emph{unique} morphism $\widetilde{f}:c\to \widetilde{d}$ in $\C_1$ with $p(\widetilde{f})=f$. 
		
		A \emph{morphism of discrete fibrations} $f$ from $p:\C_1\to \C_0$ to $q:\D_1\to \D_0$ is a commutative diagram
		\[
		\begin{tikzcd}
			\C_1 \arrow[r,"f_1"]\arrow[d,"p"'] & \D_1\arrow[d,"q"]\\
			\C_0 \arrow[r,"f_0"'] & \D_0 
		\end{tikzcd}
		\]  
		A \emph{2-morphism} $\mu$ between morphisms $f,g:p\to q$ of discrete fibrations consists of a pair of natural transformations  $\mu_1:f_1\to g_1$ and $\mu_2:f_2\to g_2$ such that the diagram
		\[
		\begin{tikzcd}[row sep=5em,column sep=5em]
			\C_1 \arrow[r,bend left,"f_1",""'{name=U1}]\arrow[r,bend right,"g_1"',""{name=L1}]\arrow[d,"p"'] & \D_1\arrow[d,"q"]\\
			\C_0 \arrow[r,bend left,"f_0",""'{name=U0}]\arrow[r,bend right,"g_0"',""{name=L0}] & \D_0 \arrow[from=U1,to=L1, Rightarrow, "\mu_1"]\arrow[from=U0,to=L0, Rightarrow, "\mu_0"]
		\end{tikzcd}
		\]
		commutes, in the sense that there is an equality of whiskerings  $q\circ \mu_1=\mu_0\circ p$. 
		
		Discrete fibrations with their morphisms and 2-morphisms form a 2-category which we denote as $\DFib$. 
	\end{defn}

	\begin{rmk}
		Unlike fibrations in groupoids or categories, we need not define a cleavage (as in \cite[\S 2.1]{MoellerVasilakopoulou}) on a discrete fibration, as every such fibration has a unique cleavage associated to it. 
	\end{rmk}
	
	\begin{defn}
		A \emph{indexed set} is a functor $F:\I\to \Set$ from a small category $\I$ to the category of sets. A \emph{morphism of indexed sets} from $F:\I\to \Set$ to $G:\J\to \Set$ is a functor $M:\I\to \J$ together with a natural transformation $\mu:F\Rightarrow G\circ M$, i.e., such that $\mu$ fills the diagram 
		\[
		\begin{tikzcd}
			\I\arrow[dr,"F",""'{name=U}]\arrow[dd,"M"'] & \\
			 & \Set \\
			 \J\arrow[ur,"G"'] & \arrow[from=U,to=3-1,Rightarrow,"\mu",shorten >= 0.5em]
		\end{tikzcd}	
		\]
		 A \emph{2-morphism} between two such one morphisms $(M,\mu)$ and $(N,\nu)$ is a 2-morphism $\eta: M\Rightarrow N$ such that $(G\circ \eta)\star \mu=\nu$. 
		 
		Indexed sets with their morphisms and 2-morphisms form a 2-category which we denote as $\ISet$. 
	\end{defn}

	We now define the pair of pseudofunctors which constitute the Grothendieck construction. We begin with the strict 2-functor 
	\[
	\begin{tikzcd}
		{\displaystyle\int}:&[-3em] \ISet \arrow[r] & \DFib. 
	\end{tikzcd}
	\]
	Given an indexed category $F:\I\to \Set$, we define the discrete fibration $\int_{\I}F\to \I$ as follows. The objects of $\int_\I F$ are pairs $(i,x)$ consisting of $i\in \I$ and $x\in F(i)$, and a morphism $f:(i,x)\to (j,y)$ is a morphism $f:i\to j$ in $\I$ such that $F(f)(x)=y$. Given a morphism 
	\[
	\begin{tikzcd}
		\I\arrow[dr,"F",""'{name=U}]\arrow[dd,"M"'] & \\
		& \Set \\
		\J\arrow[ur,"G"'] & \arrow[from=U,to=3-1,Rightarrow,"\mu",shorten >= 0.5em]
	\end{tikzcd}	
	\]
	of indexed categories, we define $\int(M)$ to be the functor which sends $(i,x)$ to $(M(i),\mu_i(x))$ and sends $f:(i,x)\to (j,y)$ to $M(f)$. The corresponding morphism in $\DFib$  is then the commutative diagram
	\[
	\begin{tikzcd}
		{\displaystyle\int_{\I}}F\arrow[d] \arrow[r,"\int(M)"] &  {\displaystyle\int_{\J}}G\arrow[d]\\
		\I \arrow[r,"M"'] & \J
	\end{tikzcd}
	\]
	
	Given an 2-morphism $\eta:M\Rightarrow N$ from $(M,\mu)$ to $(N,\nu)$, we define a 2-morphism $\eta_1:\int(M)\Rightarrow \int(N)$ as follows. The component $(\eta_1)_{(i,x)}:(M(i),\mu_i(x))\to (N(i),\mu_i(x))$ is given by $\eta_i:M(i)\to N(i)$ in $\J$. 
	
	It is easy to check that this construction preserves both composition and identities strictly, and thus defines a 2-functor. 
	
	On the other hand, we define the 2-functor 
	\[
	\begin{tikzcd}
		\on{T}: &[-3em] \DFib \arrow[r] & \ISet
	\end{tikzcd}
	\]	
	as follows. 
	
	Given a discrete fibration $p:\C_1\to \C_0$, we define an indexed set $\on{T}(p):\C_0\to \Set$ which sends $c\mapsto p^{-1}(c)$, and for $f:c\to d$ and $x\in p^{-1}(c)$, we define $\on{T}(p)(f)(x)$ to be the unique object of $p^{-1}(d)$ such that there is a morphism $\tilde{f}:x\to \on{T}(p)(f)(x)$ over $f$ in $\C_1$. 
	
	Given a morphism of discrete fibrations 
	\[
	\begin{tikzcd}
		\C_1 \arrow[r,"f_1"]\arrow[d,"p"'] & \D_1\arrow[d,"q"]\\
		\C_0 \arrow[r,"f_0"'] & \D_0 
	\end{tikzcd}
	\]
	we define a natural transformation $\on{T}(f):\on{T}(p)\Rightarrow \on{T}(q)\circ f_0$ to have component at $c$ given by $f_1|_{p^{-1}(c)}:p^{-1}(c)\to q^{-1}(f_0(c))$.
	
	Finally, given a 2-morphism $(\mu_0,\mu_1):f\Rightarrow g$ between morphisms $f,g:p\to q$, we define $\on{T}(\mu)$ to be the natural transformation $\mu_0: f_0\Rightarrow g_0$. 
	
	It is again easily checked that this defines a strict 2-functor. 
	
	\begin{defn}
		A \emph{strict 2-equivalence} between strict 2-categories $\bb{B}$ and $\bb{C}$ is a pair of strict 2-functors $F:\bb{B}\to \bb{C}$ and $G:\bb{C}\to \bb{B}$ together with strict 2-natural isomorphisms $G\circ F\cong \on{id}$ and $F\circ G\cong \on{id}$. Equivalently, this is a equivalence of $\Cat$-enriched categories (see, for instance, Definition 6.2.10, Lemma 6.2.12, and Theorem 7.5.8 of \cite{Johnson_Yau_2021}).
	\end{defn}
	
	\begin{thm}\label{thm:classGC}
		The functors $\int$ and $\on{T}$ form a strict 2-equivalence between $\DFib$ and $\ISet$. 
	\end{thm}

	\begin{proof}
		We define two strict natural transformations as follows. First, $\Phi:\on{T}\circ \int\Rightarrow \on{Id}_{\ISet}$. Given an indexed set $F:\I\to \Set$, the indexed set $\on{T}\left(\int_\I F\right)$ is the functor which sends $i\in \I$ to the set 
		\[
		\{(i,x)\mid x\in F(x)\}
		\]
		The component $\Phi_F$ of the natural transformation is thus the natural isomorphism with components 
		\[
		\begin{tikzcd}[row sep=0em]
			(\Phi_F)_i: &[-3em] \on{T}\left(\int_\I F\right)(i)\arrow[r] & F(i) \\
			&(i,x) \arrow[r,mapsto] & x.
		\end{tikzcd}
		\]
		It is not hard to check that this is a strict 2-natural transformation.
		
		On the other hand, we define a strict 2-natural transformation 
		\[
		\begin{tikzcd}
			\Psi: &[-3em] \int\circ \on{T} \arrow[r,Rightarrow] & \on{Id}_{\DFib}
		\end{tikzcd}
		\]
		as follows. Given a discrete fibration $p:\C_1\to \C_0$, the total space of the discrete fibration $\int_{\C_0}(\on{T}(p))\to \C_0$ has objects $(c,x)$ given by pairs $c\in \C_0$ and $x\in p^{-1}(c)$ and morphisms $f:(c,x)\to (d,y)$ given by morphisms $f:c\to d$ such that the unique lift $\widetilde{f}$ of $f$ starting from $x$ has target $y$. The component $\Psi_p$ is thus the functor which sends $(c,x)\mapsto x$, and $f:(c,x)\to (d,y)$ to $\widetilde{f}$. Once again, 2-naturality is easy to check. 
	\end{proof}
	
	\section{$\O$-pseudomonoids in 2-categories}\label{sec:opsmon}
	
	By the word \emph{operad}, we will always mean a symmetric, monochromatic operad in $\Set$. The terminal operad is thus the commutative operad $\scr{C}\!\on{omm}$. We will follow \cite{CornerGurski} in denoting the composition maps in an operad $\scr{O}$ by $\mu$'s when $\scr{O}$ is clear from context, and by $\mu^{\scr{O}}$'s otherwise. We will similarly write $\eta$ or $\eta^{\scr{O}}$ for the unit.
	
	In a 2-category (like, say, $\Cat$) we denote the horizontal composition by $\circ$ and the vertical composition by $\star$. We denote the terminal category by $\ast$. 
	
	\begin{rmk}
		Our notion of $\scr{O}$-pseudomonoid is not identical to either that of \cite{CornerGurski} or that of \cite{GMMO}. Since our aim is to study monoidal categories, the strict identity of \cite{GMMO} is the proper framework, however, the \emph{unit object axiom} \cite[Axiom 2.17]{GMMO} is ill-suited for our purpose, since it functions in the case $\scr{O}(0)=\ast$, and requires \emph{strict} monoidal unitality. As a result, we spell out our definitions below. 
		
		The other difference between our approach and that of \cite{GMMO} is that, as we aim to study \emph{lax monoidal functors}, it is convenient for us to reverse the directions of all of the natural transformations in the definitions of \cite{GMMO}, as a na\"ive laxening of their definitions would yield \emph{oplax} monoidal functors.  
	\end{rmk}
	
	Throughout, we fix a 2-category $\bb{B}$ with finite strict 2-products and with a strictly fully-faithful, product-preserving functor $\Set\to \bb{B}$, via which we will consider $\Set$ as a subcategory of $\bb{B}$.
	
	\begin{example}\
		\begin{enumerate}
			\item The 2-category $\Cat$ has strict 2-products given by the Cartesian product of categories. We consider $\Set$ as the full subcategory of $\Cat$ on the discrete categories. 
			\item The 2-category $\DFib$ has finite strict 2-products. The product of $p_j:\C_1^j\to \C_0^j$ in $\DFib$ is given by the product 
			\[
			\begin{tikzcd}
				\prod_j p_j: &[-3em] \prod_j \C_1^j \arrow[r] & \prod_j \C_0^j 
			\end{tikzcd}
			\]
			in $\Cat$. 
			
			We consider a set $X$ as the element $\id_X:X\to X$ of $\DFib$, providing the desired full subcategory. 
			\item The 2-category $\ISet$ has finite strict 2-products. The product of $\{F_j:\I_j\to \Set\}_{j=1}^n$  is the functor 
			\[
			\begin{tikzcd}
				\prod_{j} F_j :&[-3em] \prod_j \I_j \arrow[r] & \Set 
			\end{tikzcd}
			\] 
			given by composing the product of the functors in $\Cat$ with the functor $\Set^{\times n}\to \Set$ sending a tuple of sets to its Cartesian product. Note that for $F:\I\to \Set$, $X\times F$ can be canonically identified with the functor$X\times \I \to \I \to \Set$ which composes $F$ with the projection.
			
			We consider a set $X$ as the functor $X\to \Set$ from the discrete category $X$ to $\Set$ which is constant on the terminal object. This provides the desired full subcategory.
		\end{enumerate}
	\end{example}

	 We further denote by $\Fin$ the skeleton of the category of finite (possibly empty) sets with objects $\underline{n}:=\{1,\ldots,n\}$. The following definition of an operad is more in the vein of \cite[Definition 2.1.1.1]{HA}, as this formulation simplifies the abstract arguments of the paper. 
	
	\begin{defn}
		An \emph{operad} $\scr{O}$ consists of, for every $n\geq 0$, a set $\scr{O}(\underline{n})\in \Set$, which we abusively denote by $\scr{O}(n)$, and, for every map $f:\underline{n}\to \underline{m}$ in $\Fin$, a map 
		\[
		\begin{tikzcd}
			\mu_f:&[-3em] \scr{O}(m)\times {\displaystyle \prod_{i\in \underline{m}}}\left( \scr{O}(|f^{-1}(i)|)\right)\arrow[r] & \scr{O}(n)
		\end{tikzcd}
		\]
		satisfying the following two conditions. 
		\begin{enumerate}
			\item For every composable pair of maps 
			\[
			\begin{tikzcd}
				\underline{\ell}\arrow[r,"g"] & \underline{m}\arrow[r,"f"] & \underline{n} 
			\end{tikzcd}
			\]
			in $\Fin$, the diagram 
			\[
			\begin{tikzcd}
				\scr{O}(n) \times {\displaystyle\prod_{i\in \underline{n}}}\left( \scr{O}(|f^{-1}(i)|)\times\left({\displaystyle}\prod_{j\in f^{-1}(i)}\scr{O}(|g^{-1}(j)|)\right)\right) \arrow[r,"\prod_{i\in \underline{n}} \mu_{g_i}"] \arrow[d,"\mu_f"'] & \scr{O}(n)\times {\displaystyle\prod_{i\in \underline{n}}}\left( \scr{O}(|(f\circ g)^{-1}(i)|)\right)\arrow[d,"\mu_{f\circ g}"]\\
				\scr{O}(\underline{m})\times {\displaystyle\prod_{j\in \underline{m}}}\left( \scr{O}(|g^{-1}(j)|)\right) \arrow[r,"\mu_g"'] & \scr{O}({\ell})
			\end{tikzcd}
			\]
			commutes, where $g_i$ denotes the map $|(f\circ g)^{-1}(i)|\to |f^{-1}(i)|$ induced by $g$, and we have suppressed instances of the structure maps in $\Set$. 
			\item There is an element
			\[
			\begin{tikzcd}
				\eta: &[-3em] \ast \arrow[r] & \scr{O}(1)
			\end{tikzcd}
			\]
			such that the diagrams 
			\[
			\begin{tikzcd}
				 \scr{O}( {n})\times \ast \arrow[r,"\on{id}\times \eta^n"]\arrow[dr,"\on{id}"'] & \scr{O}({n})\times \prod_{i\in \underline{n}} (\scr{O}(1))\arrow[d,"\mu_{\on{id}_{\underline{n}}}"] \\
				& \scr{O}(n) 
			\end{tikzcd}
			\]
			and 
			\[
			\begin{tikzcd}
				\ast \times \scr{O}(n)\arrow[r,"\eta\times \on{id}"]\arrow[dr,"\on{id}"'] & \scr{O}(1)\times \scr{O}(n)\arrow[d,"\mu_{t_n}"] \\
				 & \scr{O}(n) 
			\end{tikzcd}
			\]
			commute, where $t_n$ denotes the unique map $\underline{n}\to \underline{1}$. 
		\end{enumerate}
	\end{defn}

	\begin{example}\ \label{example : comm, assoc}
		\begin{enumerate}
			\item It is immediate that there is a terminal operad, the \emph{commutativity operad} $\Comm$, with $\Comm(n)=\ast$. 
			\item Define the \emph{associativity operad} $\Assoc$ by setting $\Assoc(n):=\Sigma_n$ to be the group of permutations of $n$ letters, viewed as $\on{Aut}_{\Fin}(\un)$. Given $f:\um\to \un$, and a permutation $\sigma$ of $\un$, there are a unique order-preserving map $f_\sigma:\um\to \un$ and a unique permutation $\sigma_f:\um\to \um$ such that the diagram 
			\[
			\begin{tikzcd}
				\um\arrow[d,"f"'] \arrow[r,"\sigma_f"] & \um\arrow[d,"f_{\sigma}"] \\
				\un \arrow[r,"\sigma"'] & \un 
			\end{tikzcd}
			\]
			commutes and $\sigma_f$ preserves the orders on fibres induced by the order on $\underline{m}$.

			We then define, for $f:\um\to \un$, the composition map $\mu_f$ by
			\[
			\mu_f(\sigma;\tau_1,\ldots,\tau_n)=\sigma_f\circ (\tau).
			\]
			for $\tau_i$ a permutation of $f^{-1}(i)$, and $\tau=T(\{\tau_i\}_\{i\in\un\})$ the permutation of $\um$ which operates on the fibre over $i$ as $\tau_i$. Note that we can consider the collection $\tau_1,\ldots,\tau_n$ equivalently as a permutation $T(\{\tau_i\}_\{i\in\un\})$ such that the diagram 
			\[
			\begin{tikzcd}
				\um \arrow[dr,"f"']\arrow[rr,"T(\{\tau_i\}_\{i\in\un\})"] & & \um\arrow[dl,"f"] \\
				 & \um & 
			\end{tikzcd}
			\]
			commutes.

			The associativity condition then requires that, for any $\begin{tikzcd}
				\underline{\ell}\arrow[r,"g"] & \underline{m}\arrow[r,"f"] & \underline{n}
			\end{tikzcd}$, $\sigma \in \Sigma_n$, $\tau_i\in \Sigma_{f^{-1}(i)}$, and $\xi_{i,j}\in \Sigma_{g^{-1}(j)}$ for each $j\in f^{-1}(i)$ 
			\[
			[\sigma_f\circ T(\{\tau_i\}_i)]_g\circ T(\{\xi_{i,j}\}_{i,j})=\sigma_{f\circ g}\circ T\left(\left\lbrace(\tau_i)_{g_i}\circ T(\{\xi_{i,j}\}_{j\in f^{-1}(i)})\right\rbrace_{i}\right).
			\]
			It is not hard to see that 
			\[
			T(\{\alpha_i\circ \beta_i\}_i)=T(\{\alpha_i\}_i)\circ T(\{\beta_i\}_{i})
			\]
			and 
			\[
			T(\{\alpha_{i,j}\}_{i,j})=T\left(\{T(\{\alpha_{i,j}\}_j)\}_i\right)
			\]
			So the final term on both sides is $T(\{\xi_{i,j}\}_{i,j})$. It thus suffices to consider the case where the $\xi_{i,j}$ are identities Thus, we are left to show that 
			\[
			[\sigma_f\circ T(\{\tau_i\}_i)]_g=\sigma_{f\circ g}\circ T\left(\left\lbrace(\tau_i)_{g_i}\right\rbrace_{i}\right).
			\]
			We then note that 
			\[
			T(\{(\tau_i)_{g_i}\}_i)=T(\{\tau_i\}_i)_g.
			\]
			Thus, we are left to show that 
			\[
			[\sigma_f\circ \tau]_g=\sigma_{f\circ g}\circ \tau_g. 
			\]
			However, applying the uniqueness of $\sigma_f$ to the diagram 
			\[
			\begin{tikzcd}
				\ul\arrow[rr,bend left, "(\sigma_f\circ \tau)_g"]  \arrow[r,"\tau_g"]\arrow[d,"g"'] &\ul \arrow[r,"\sigma_{f\circ g}"] \arrow[d,"g_\tau"]&\ul\arrow[d,"g_{\sigma_f}"] \\
				\um \arrow[r,"\tau"'] & \um \arrow[r,"\sigma_f"] \arrow[d,"f"] & \um\arrow[d,"f_\sigma"] \\
				 & \un \arrow[r, "\sigma"'] &\un 
			\end{tikzcd}
			\]
			puts paid to this check. 
			
			It is immediate that the unitality conditions are satisfied, making $\Assoc$ an operad. 
			\item Given a semiring $R$, there is a operad $\QConv_R$ defined by setting $\QConv_R(n)$ to be the set
			\[
			\left\lbrace(\alpha_1,\ldots,\alpha_n)\in R^n\mid \sum_{i=1}^n\alpha_i=1\right\rbrace
			\]
			The composition for $f:\um\to \un$ is defined by 
			\[
			\mu_f(\vec(\alpha);\vec{\beta}^1,\ldots, \vec{\beta}^n)=\left(\alpha_{f(1)}\beta^{f(1)}_1,\ldots,\alpha_{f(1)}\beta^{f(1)}_{|f^{-1}(f(1))|},\ldots, \alpha_{f(m)}\beta^{f(m)}_{1},\ldots,\alpha_{f(m)}\beta^{f(m)}_{|f^{-1}(f(m))|}\right)
			\]
			Note that since there a no nullary operations (i.e., $\QConv_R(0)=\varnothing$), we may without loss of generality restrict the composition to surjective maps. See \cite{HOSconvex} for more details.
		\end{enumerate}
	\end{example}

	Given an isomorphism $f:\underline{n}\to \underline{m}$ in $\Set$, denote by $f^\ast: \C^{\underline{m}}\to \C^{\underline{n}}$ the induced map of strict 2-products.

	\begin{defn}\label{defn:pseudomonoid}
		Given an operad $\scr{O}$, an $\scr{O}$-pseudomonoid in $\bb{B}$ is an object $\C$ of $\bb{B}$ equipped with 1-morphisms 
		\[
		\begin{tikzcd}
			\gamma_n:&[-3em] \scr{O}(n)\times \sf{C}^{\underline{n}} \arrow[r] & \sf{C}
		\end{tikzcd}
		\]
		and, for every $f:\underline{m}\to \underline{n}$ in $\Fin$, 2-isomorphisms 
		\[
		\begin{tikzcd}[column sep=5em]
			\scr{O}(n)\times {\displaystyle\prod_{i\in\underline{n}}}\left(\scr{O}(|f^{-1}(i)|)\times \sf{C}^{ f^{-1}(i)}\right)\arrow[r,"{\on{id}\times \prod \gamma_{f^{-1}(i)}}",""'{name=U}]\arrow[d,"(\mu_f\times \on{id})\circ \tau"'] & \scr{O}(n)\times \sf{C}^{\underline{n}} \arrow[d,"\gamma_n"]\\
			\scr{O}\left(m\right)\times  \sf{C}^{\underline{m}} \arrow[r,"{\gamma_{m}}"',""{name=R}] & \sf{C} \arrow[from=U,to=R, Leftarrow,shorten >=3ex, shorten <=3ex,"\phi_{f}"] 
		\end{tikzcd}
		\]
		satisfying the following four conditions:
		\begin{enumerate}
			\item (associativity) For any maps 
			\[
			\begin{tikzcd}
				\underline{\ell}\arrow[r,"g"] & \underline{m}\arrow[r,"f"] & \underline{n} 
			\end{tikzcd}
			\]
			in $\Fin$, the two pasting diagrams 
			\[
			\begin{tikzcd}[column sep=2em,every matrix/.append style = {font = \scriptsize},every label/.append style = {font = \tiny}]
				 \scr{O}(n)\times {\displaystyle \prod_{i\in \underline{n}}}\left[\scr{O}(f^{-1}(i))\times \!\!\!\!\!{\displaystyle \prod_{j\in f^{-1}(i)}}\scr{O}(g^{-1}(j))\C^{g^{-1}(j)}\right] \arrow[rr,"\id\times\Pi_i(\id \times \Pi_j \gamma_{g^{-1}(j)})"]\arrow[dd,"(\mu_f\times \id)\circ \tau"'] & & \scr{O}(n)\times {\displaystyle \prod_{i\in \underline{n}}}\left[\scr{O}(f^{-1}(i))\times \C^{f^{-1}(j)}\right]\arrow[dd,"(\mu_f\times \id)\circ \tau"'] \arrow[dr,"\id\times \prod_i \gamma_{f^{-1}(i)}"] & \\
				 & & & \scr{O}(n)\times \C^{\underline{n}}\arrow[dd,"\gamma_n"]\\
				 \scr{O}(m) \times {\displaystyle \prod_{\substack{i\in \underline{n}\\ j\in f^{-1}(i)}}}\left[\scr{O}(g^{-1}(j))\times \C^{g^{-1}(j)}\right]\arrow[rr,"\id\times \Pi_{i,j}\gamma_{g^{-1}(j)}"]\arrow[dr,"(\mu_g\times\id)\circ \tau"'] & & \scr{O}(m)\times \C^{\underline{m}}\arrow[dr,"\gamma_m"] \arrow[ur,Rightarrow,"\phi_f"]\\ 
				 & \scr{O}(\ell)\times \C^{\underline{\ell}} \arrow[rr,"\gamma_\ell"']\arrow[ur,Rightarrow,"\phi_g"] & & \C 
			\end{tikzcd}
			\]
			and
			\[
			\begin{tikzcd}[column sep=-3em,every matrix/.append style = {font = \scriptsize},every label/.append style = {font = \tiny}]
				\scr{O}(n)\times {\displaystyle \prod_{i\in \underline{n}}}\left[\scr{O}(f^{-1}(i))\times \!\!\!\!\!{\displaystyle \prod_{j\in f^{-1}(i)}}\scr{O}(g^{-1}(j))\C^{g^{-1}(j)}\right] \arrow[rr,"\id\times\Pi_i(\id \times \Pi_j \gamma_{g^{-1}(j)})"]\arrow[dd,"(\mu_f\times \id)\circ \tau"']\arrow[dr,"\id\times((\mu_g\times \id)\circ \tau)"'] & & \scr{O}(n)\times {\displaystyle \prod_{i\in \underline{n}}}\left[\scr{O}(f^{-1}(i))\times \C^{f^{-1}(j)}\right] \arrow[dr,"\id\times \prod_i \gamma_{f^{-1}(i)}"] &[4em] \\
				&\scr{O}(n)\times {\displaystyle\prod_{i\in{\underline{n}}}}\left[ \scr{O}((fg)^{-1}(i))\times \C^{{(fg)}^{-1}(i)}\right]\arrow[dd,"(\mu_{fg}\times \id)\circ \tau"']\arrow[rr,"\id\times \Pi_i\gamma_{(fg)^{-1}(i)}"']\arrow[ur,Rightarrow,"\on{Id}\times \Pi_i\phi_{g_i}"] & & \scr{O}(n)\times \C^{\underline{n}}\arrow[dd,"\gamma_n"]\\
				\scr{O}(m) \times {\displaystyle \prod_{\substack{i\in \underline{n}\\ j\in f^{-1}(i)}}}\left[\scr{O}(g^{-1}(j))\times \C^{g^{-1}(j)}\right]\arrow[dr,"(\mu_g\times\id)\circ \tau"'] & &  \\ 
				& \scr{O}(\ell)\times \C^{\underline{\ell}} \arrow[rr,"\gamma_\ell"']\arrow[uurr,Rightarrow,"\phi_{fg}"] & & \C 
			\end{tikzcd}
			\]
			are equal. In these diagrams, we have abusively written $\tau$ for any map induced by a transformation of product diagrams which consists of a bijection onto elements not sent to the terminal object.
			\item (Identity 1) The diagram 
			\[
			\begin{tikzcd}
				\C\arrow[d,"\cong"']\arrow[ddr,"\on{id}",""'{name=U}] & \\
				\ast\times \C\arrow[d,"\eta\times \on{id}"'] & \\
				\scr{O}(1)\times \C \arrow[r,"\gamma_1"'] & \C
			\end{tikzcd}
			\]
			commutes. 
			\item (Identity 2) For any $n\geq 1$, 
			\[
			\phi_{\on{id}_{\underline{n}}}\circ \left(\on{id}\times(\eta\times \on{id})^n\right)=\on{Id}_{\gamma_n}
			\]
			where 
			\[
			\begin{tikzcd}
				(\on{id}\times(\eta\times \on{id})^n):&[-3em] \scr{O}(n)\times \sf{C}^{\times n} \arrow[r] & \scr{O}(n)\times \left(\scr{O}(1)\times \sf{C}\right)^{\times n}.
			\end{tikzcd}
			\]
			\item (Identity 3) For any $n\geq 0$, 
			\[
			\phi_{t_n}\circ (\eta\times \on{id})= \on{Id}_{\gamma_n},
			\]
			where 
			\[
			\begin{tikzcd}
				(\eta\times \on{id}):&[-3em] \scr{O}(n)\times \sf{C}^n \arrow[r] & \scr{O}(1)\times \scr{O}(n)\times \sf{C}^n 
			\end{tikzcd}
			\] 
		\end{enumerate}
		We will call an $\scr{O}$-pseudomonoid \emph{strict} if all of the $\phi_i$ are identities, recovering the notion of a monoid in the underlying 1-category. 
	\end{defn}

	\begin{rmk}
		This definition may seem odd to those familiar with the usual definition of, e.g., symmetric monoidal categories, which define a braiding entirely separately from the associators. However, as we will see when we draw explicit connections to symmetric monoidal categories, the braidings in this setting are the $\phi_f$ associated to permutations $f:\underline{n}\to \underline{n}$. 
	\end{rmk}
	
	\begin{defn}\label{defn:O-lax-morph}
		Let $(\sf{C},\gamma,\phi)$ and $(\sf{D},\lambda,\psi)$ be two $\scr{O}$-pseudomonoids in $\bb{B}$. A \emph{lax $\scr{O}$-morphism} from $\sf{C}$ to $\sf{D}$ is a 1-morphism 
		\[
		\begin{tikzcd}
			F: &[-3em] \sf{C}\arrow[r] & \sf{D}
		\end{tikzcd}
		\]
		together with 2-morphisms 
		\[
		\begin{tikzcd}
			\scr{O}(n)\times \sf{C}^{\underline{n}} \arrow[r,"\gamma_n"]\arrow[d,"\on{id}\times F"'] & \sf{C}\arrow[d,"F"] \\
			\scr{O}(n)\times \sf{D}^{\underline{n}} \arrow[r,"\lambda_n"'] & \sf{D} 
			\arrow[from=1-2, to=2-1, Leftarrow,"\xi_n"]
		\end{tikzcd}
		\] 
		for all $n\geq 0$, which satisfy the first condition in \cite[Definition 2.4]{CornerGurski}, and further satisfy that the 2-morphism 
		\[
		\begin{tikzcd}
			\sf{C} \arrow[r,"\cong"]\arrow[d,"F"'] & \ast\times \sf{C} \arrow[r,"\eta\times \on{id}"]\arrow[d,"\on{id}\times F"'] & \scr{O}(1)\times \sf{C} \arrow[r,"\gamma_1"]\arrow[d,"\on{id}\times F"'] & \sf{C}\arrow[d,"F"] \\
			\sf{D} \arrow[r,"\cong"'] & \ast\times \sf{D} \arrow[r,"\eta\times \on{id}"'] & \scr{O}(1)\times \sf{D} \arrow[r,"\lambda_1"'] & \sf{D} \arrow[from=2-3, to=1-4,Rightarrow,"\xi_1"] 
		\end{tikzcd}
		\]
		is the identity.
		
		Finally, the 1-morphism $F$ and $2$-morphisms $\xi$ must satisfy the condition that, for any $f:\underline{m}\to \underline{n}$, the front (red) and back (blue) of the pasting diagram 
		\[
		\begin{tikzcd}[column sep=2em,row sep=3em,every matrix/.append style = {font = \scriptsize},every label/.append style = {font = \tiny}]
			\scr{O}(n)\times {\displaystyle \prod_{i\in \underline{n}}}\left(\scr{O}(f^{-1}(i))\times \C^{f^{-1}(i)}\right) \arrow[rr,"\id\times\Pi_i(\id \times F^{f^{-1}(i)})"]\arrow[dd,"(\mu_f\times \id)\circ \tau"']\arrow[dr,red,"\id\times\Pi_i\gamma_{f_(i)}"] & & \scr{O}(n)\times {\displaystyle \prod_{i\in \underline{n}}}\left(\scr{O}(f^{-1}(i))\times \D^{f^{-1}(i)}\right)\arrow[dd,blue,"(\mu_f\times \id)\circ \tau"'{pos=0.6}] \arrow[dr,"\id\times \prod_i \lambda_{f_i}"] & \\
			& {\color{red} \scr{O}(n)\times \C^{\underline{n}}}\arrow[rr,red,"\id\times F^{\underline{n}}"{pos=0.7}]\arrow[dd,red,"\gamma_n"{pos=0.6}]\arrow[ur,red,Rightarrow,"\id\times \Pi_i \xi_{f^{-1}(i)}"']& & \scr{O}(n)\times \D^{\underline{n}}\arrow[dd,"\lambda_n"]\\
			\scr{O}(m) \times C^{\underline{m}}\arrow[rr,blue,"\id\times F"{pos=0.6}]\arrow[dr,"\gamma_m"']\arrow[ur,red,Rightarrow,"\phi_f"] & & {\color{blue}\scr{O}(m)\times \D^{\underline{m}}}\arrow[dr,blue,"\lambda_m"] \arrow[ur,blue,Rightarrow,"\phi_f"]\\ 
			& \C \arrow[uurr,red,Rightarrow,bend right=1em,"\xi_n"'] \arrow[rr,"F"']\arrow[ur,blue,Rightarrow,"\xi_m"] & & \D 
		\end{tikzcd}
		\] 
		are equal, i.e., that the cube commutes.
		
		We call a lax $\scr{O}$-morphism a \emph{weak $\scr{O}$-morphism} or a \emph{strict $\scr{O}$-morphism} if all of the $\xi_n$ are isomorphisms or identities, respectively. 
	\end{defn}

	\begin{rmk}\label{rmk:assoc_alg_set}
		If $\bb{B}$ is a 1-category, considered as a 2-category with discrete hom-categories, the definitions above collapse to the usual notion of algebra over an operad and morphism of algebras. Since our presentation of operads is somewhat unconventional (though it appears in, e.g., \cite[2.1.1.1]{HA}) and makes the role of the permutations rather opaque, let us unwind the meaning of an $\Assoc$-algebra in $\Set$. If $(A,\gamma)$\footnote{Since the $\phi$'s necessarily collapse to identities, we omit them here.} is an $\Assoc$-algebra in $\Set$, then we have, for every permutation $\sigma$ of $\un$, a map 
		\[
		\begin{tikzcd}
			m_\sigma:=(\gamma_n)|_{\{\sigma\}\times A}: &[-3em] A^{\un}\arrow[r] & A.
		\end{tikzcd}
		\]
		subject to the condition that the diagram defining $\phi_f$ commutes for any $f:\um\to \un$.
		
		However, supposing that $\sigma: \un\to \un$ is a permutation, and letting $(a_1,\ldots,a_n)$ be an $\underline{n}$-indexed collection of elements in $A$ (which we think of as associated to the \emph{source} of $\sigma$),the commutativity of the diagram 
		\[
		\begin{tikzcd}[column sep=5em]
			\Assoc(n)\times {\displaystyle\prod_{i\in\underline{n}}}\left(\Assoc(|\sigma^{-1}(i)|)\times A^{ \sigma^{-1}(i)}\right)\arrow[r,"{\on{id}\times \prod \gamma_{\sigma^{-1}(i)}}",""'{name=U}]\arrow[d,"(\mu_\sigma\times \on{id})\circ \tau"'] & \Assoc(n)\times A^{\underline{n}} \arrow[d,"\gamma_n"]\\
			\Assoc\left(m\right)\times  A^{\underline{m}} \arrow[r,"{\gamma_{m}}"',""{name=R}] & A 
		\end{tikzcd}
		\]
		when specialized to the sequence $(a_1,\ldots,a_n)$ and the operation $\id_{\un}\in \Assoc(n)$, yields 
		\[
		\begin{tikzcd}
			\{\id\}\times (\{\eta\}\times a_{\sigma^{-1}(i)})_{i\in\un} \arrow[r,mapsto]\arrow[d,mapsto] & \{\id\}\times (a_{\sigma^{-1}(1)},\ldots, a_{\sigma^{-1}(n)}) \\
			\{\sigma \}\times(a_1,\ldots,a_n)  
		\end{tikzcd}
		\]
		and so we see that 
		\[
		m_\sigma(a_1,\ldots,a_n)=m_{\id_{\un}}(a_{\sigma^{-1}(1)},\ldots,a_{\sigma^{-1}(n)}).
		\]
		That is, there is really only one $n$-ary operation in play. The analysis of these $n$-ary operations can then be performed using the order-preserving maps, showing, in particular, that 
		\[
		m_{\id_\un}\circ (m_{\id_{\um_1}},\ldots,m_{\id_{\um_n}})=m_{\id_{\sum_i \um_i}}.
		\]
		The unit axioms amount to requiring that $m_{\id_{\underline{1}}}=\id$, and so we see that our definition of an algebra yields the usual (unbiased) notion of an associative algebra. 
	\end{rmk}
	
	\begin{defn}
		Let $(\sf{C},\gamma,\phi)$ and $(\sf{D},\lambda,\psi)$ be two $\scr{O}$-pseudomonoids and $(F,\xi),(G,\zeta):\sf{C}\to \sf{D}$ two lax  $\scr{O}$-pseudomonoids. A \emph{$\scr{O}$-2-morphism} from $F$ to $G$ is a 2-morphism $\mu:F\Rightarrow G$ such that the pasting diagrams 
		\[
		\begin{tikzcd}
			\scr{O}(n)\times \sf{C}^n \arrow[r,"\on{id}\times F^n"]\arrow[d,"\gamma_n"'] & \scr{O}(n)\times \sf{D}^n\arrow[d,"\lambda_n"]\arrow[dl,Rightarrow,"\xi_n"'] & & \scr{O}(n)\times \sf{C}^n \arrow[r,"\on{id}\times G^n"',""{name=B}] \arrow[r,bend left=4em, "\on{id}\times F^n",""'{name=A}] \arrow[d,"\gamma_n"'] & \scr{O}(n)\times \sf{D}^n\arrow[d,"\lambda_n"]\arrow[dl,Rightarrow,"\zeta_n"']\\
			\sf{C}\arrow[r,"F",""'{name=U}]\arrow[r,bend right=4em,"G"',""{name=V}] & \sf{D} & & \sf{C}\arrow[r,"G"] & \sf{D}\arrow[from=U,to=V,Rightarrow,"\mu"]\arrow[from=A,to=B,Rightarrow,"\id\times \mu^n"']
		\end{tikzcd}
		\]
		are equal. 
	\end{defn}
	
	\begin{defn}
		There is a 2-category of $\scr{O}$-pseudomonoids, lax $\scr{O}$-morphisms, and $\scr{O}$-2-morphisms in $\bb{B}$, which we denote $\scr{O}\sf{Mon}(\bb{B})$. 
	\end{defn}

	\begin{prop}\label{prop:op_map_psmon_functor}
		If $f:\scr{O}\to \scr{P}$ is a morphism of operads, then for any $2$-category $\bb{B}$ restriction along $f$ induces a 2-functor 
		\[
		\begin{tikzcd}
			f^\ast:&[-3em] \scr{P}\sf{Mon}(\bb{B})\arrow[r] & \scr{O}\sf{Mon}(\bb{B}).
		\end{tikzcd}
		\]
	\end{prop}
	
	\begin{proof}
		Given a $\scr{P}$-pseudomonoid $(\C,\gamma,\phi)$ in $\bb{B}$, we define $f^\ast(\C,\gamma,\phi)$ to have the same underlying object $\C$, structure maps $(f^\ast \gamma)_n$ given by the composites 
		\[
		\begin{tikzcd}
			\scr{O}(n)\times\C^{\underline{n}} \arrow[r,"f_n\times \id"] & \scr{P}(n)\times \C^{\underline{n}} \arrow[r,"\gamma_n"] & \C.
		\end{tikzcd}
		\]
		The structure morphisms $(f^\ast \phi)_{g}$ are then given by the pasting diagrams 
		\[
		\begin{tikzcd}[column sep=5em]
			\scr{O}(n)\times {\displaystyle\prod_{i\in\underline{n}}}\left(\scr{O}(|g^{-1}(i)|)\times \sf{C}^{ g^{-1}(i)}\right)\arrow[dr,"f_n\times (\Pi_if_{g^{-1}(i)}\times \id)"']\arrow[dd,"(\mu_f\times \id)\circ \tau"']\arrow[rr,"\id \times \Pi_i(f^\ast \gamma)_{g^{-1}(i)}"] & & \scr{O}(n)\times \C^{\underline{n}}\arrow[d,"f_n\times \id"] \\
			& \scr{P}(n)\times {\displaystyle\prod_{i\in\underline{n}}}\left(\scr{P}(|g^{-1}(i)|)\times \sf{C}^{ g^{-1}(i)}\right)\arrow[r,"{\on{id}\times \Pi_i \gamma_{g^{-1}(i)}}",""'{name=U}]\arrow[d,"\mu_g\times \on{id}"'] & \scr{P}(n)\times \sf{C}^{\underline{n}} \arrow[d,"\gamma_n"]\\
			\scr{O}\left(m\right)\times  \sf{C}^{\underline{m}}\arrow[r,"f_m\times \id"'] & \scr{P}\left(m\right)\times  \sf{C}^{\underline{m}} \arrow[r,"{\gamma_{m}}"',""{name=R}] & \sf{C} \arrow[from=U,to=R, Leftarrow,shorten >=3ex, shorten <=3ex,"\phi_{f}"] 
		\end{tikzcd}
		\]
		It is an exercise in 3-dimensional diagram drawing to see that associativity holds for $f^\ast(\C,\gamma,\phi)$. Indeed, for each non-commutative face of the cube representing associativity, one gets either a copy of the diagram defining $(f^\ast\phi)_g$, or a product of such diagrams. This creates a cube with two ``shells'', the inner of which commutes because of the associativity for the original $\scr{P}$-pseudomonoid, and the outer commutes by construction. The unitality conditions are immediate from the fact that $f_1$ preserves the unit.
		
		Now let $(\sf{C},\gamma,\phi)$ and $(\sf{D},\lambda,\psi)$ be two $\scr{P}$-pseudomonoids in $\bb{B}$, and let $(F,\xi)$ be a lax $\scr{P}$-morphism from $\C$ to $\D$. We define a lax $\scr{O}$-morphism from $(\sf{C},f^\ast\gamma,f^\ast\phi)$ to $(\sf{D},f^\ast\lambda,f^\ast\psi)$ to have the same underlying 1-morphism $F$, and structure 2-morphisms $(f^\ast \xi)_n$ given by the pasting diagrams 
		\[
		\begin{tikzcd}
			\scr{O}(n)\times \sf{C}^{\underline{n}}\arrow[r,"f_n\times \id"]\arrow[d,"\id\times F"]& \scr{P}(n)\times \sf{C}^{\underline{n}} \arrow[r,"\gamma_n"]\arrow[d,"\on{id}\times F"'] & \sf{C}\arrow[d,"F"] \\
			\scr{O}(n)\times \sf{D}^{\underline{n}}\arrow[r,"f_n\times \id"']&\scr{P}(n)\times \sf{D}^{\underline{n}} \arrow[r,"\lambda_n"'] & \sf{D} 
			\arrow[from=1-3, to=2-2, Leftarrow,"\xi_n"]
		\end{tikzcd}
		\] 
		Unitality is again immediate from the fact that $f_1$ preserves the unit, and so we are left to check that associativity holds. This amounts to the commutativity of the outer cube in (\ref{diag:assoc_for_opfunct}). Note that the back cube commutes by construction, the bottom front cube commutes because $(F,\xi)$ is a lax morphism of $\scr{P}$-pseudomonoids, and the top front cube commutes by the construction of $f^\ast(\xi)$ and the universal property of the product.
		
		It is immediate that this assignment is strictly functorial by applying the interchange law to the pasting diagram
		\[
		\begin{tikzcd}
			\scr{O}(n)\times \sf{C}^{\underline{n}}\arrow[r,"f_n\times \id"]\arrow[d,"\id\times F"]& \scr{P}(n)\times \sf{C}^{\underline{n}} \arrow[r,"\gamma_n"]\arrow[d,"\on{id}\times F"'] & \sf{C}\arrow[d,"F"] \\
			\scr{O}(n)\times \sf{D}^{\underline{n}}\arrow[r,"f_n\times \id"']\arrow[d,"\id\times G"']&\scr{P}(n)\times \sf{D}^{\underline{n}} \arrow[r,"\lambda_n"'] \arrow[d,"\id\times G^{\underline{n}}"']& \sf{D} 
			\arrow[from=1-3, to=2-2, Leftarrow,"\xi_n"]\arrow[d,"G"]\\
			\scr{O}(n)\times \sf{E}^{\underline{n}}\arrow[r,"f_n\times \id"']&\scr{P}(n)\times \sf{E}^{\underline{n}} \arrow[r,"\\mu_n"'] & \sf{E} 
			\arrow[from=2-3, to=3-2, Leftarrow,"\zeta_n"]\\
		\end{tikzcd}
		\]
		and similarly, unitality follows immediately. 
		
		Finally, given an $\scr{P}$-2-morphism $\mu:F\Rightarrow G$ between $(F,\xi),(G,\zeta):\sf{C}\to \sf{D}$, we will show that $\mu$ itself defines a $\scr{O}$-2-morphism from $(F,f^\ast\xi)$ to $(G,f^\ast \zeta)$. We wish to show the equality of the diagrams 
		\[
		\begin{tikzcd}
			\scr{O}(n)\times \C^{\underline{n}}\arrow[d,"f_n\times \id"'] \arrow[r,"\id\times F^{\underline{n}}"] & \scr{O}(n)\times \D^{\underline{n}}\arrow[d,"f_n\times \id"] \\
			\scr{O}(n)\times \C^{\underline{n}}\arrow[d,"\gamma_n"'] \arrow[r,"\id\times F^{\underline{n}}"] & \scr{O}(n)\times \D^{\underline{n}}\arrow[d,"\lambda_n"]\\
			\C\arrow[r,"F",""'{name=U}]\arrow[r,bend right=2cm,"G"',""{name=V}] & \D\arrow[from=U,to=V, Rightarrow,"\mu"]\arrow[from=2-2,to=3-1,Rightarrow,"\xi_n"']
		\end{tikzcd} \quad \text{and}\quad 
		\begin{tikzcd}
			\scr{O}(n)\times \C^{\underline{n}}\arrow[d,"f_n\times \id"'] \arrow[r,"\id\times G^{\underline{n}}"',""{name=V}]\arrow[r,bend left=2cm,"\id\times F^n",""'{name=U}] & \scr{O}(n)\times \D^{\underline{n}}\arrow[d,"f_n\times \id"] \\
			\scr{O}(n)\times \C^{\underline{n}}\arrow[d,"\gamma_n"'] \arrow[r,"\id\times G^{\underline{n}}"] & \scr{O}(n)\times \D^{\underline{n}}\arrow[d,"\lambda_n"]\\
			\C\arrow[r,"G"] & \D\arrow[from=2-2,to=3-1,Rightarrow,"\zeta_n"']\arrow[from=U,to=V,Rightarrow,"\id\times\mu^n"]
		\end{tikzcd}
		\]
		Since $\mu$ is a $\scr{P}$-2-morphism, we need only note that the pasting diagrams 
		\[
		\begin{tikzcd}
			\scr{O}(n)\times \C^{\underline{n}}\arrow[d,"f_n\times \id"'] \arrow[r,"\id\times G^{\underline{n}}"',""{name=V}]\arrow[r,bend left=2cm,"\id\times F^n",""'{name=U}] & \scr{O}(n)\times \D^{\underline{n}}\arrow[d,"f_n\times \id"] \\
			\scr{O}(n)\times \C^{\underline{n}} \arrow[r,"\id\times G^{\underline{n}}"] & \scr{O}(n)\times \D^{\underline{n}} \arrow[from=U,to=V,Rightarrow,"\id\times \mu^n"]
		\end{tikzcd} \quad \text{and}\quad 
		\begin{tikzcd}
			\scr{O}(n)\times \C^{\underline{n}}\arrow[d,"f_n\times \id"'] \arrow[r,"\id\times F^{\underline{n}}"'] & \scr{O}(n)\times \D^{\underline{n}}\arrow[d,"f_n\times \id"] \\
			\scr{O}(n)\times \C^{\underline{n}} \arrow[r,"\id\times F^{\underline{n}}",""'{name=U}]\arrow[r,bend right=2cm,"\id\times G^n"',""{name=V}] & \scr{O}(n)\times \D^{\underline{n}} \arrow[from=U,to=V,Rightarrow,"\id\times \mu^n"]
		\end{tikzcd}
		\]
		are equal by the universal property of the product. 
	\end{proof}

	\begin{defn}
		Let $\bb{C}$ and $\bb{B}$ be 2-categories with finite strict 2-products and equipped with fully faithful functors 
		\[
		\begin{tikzcd}
			\iota_{\bb{C}}: &[-3em] \Set \arrow[r] & \bb{C}
		\end{tikzcd}
		\]
		and 
		\[
		\begin{tikzcd}
			\iota_{\bb{B}}: &[-3em] \Set \arrow[r] & \bb{B}.
		\end{tikzcd}
		\]
		We then call $(\CC,\iota_{\bb{C}})$ and $(\bb{B},\iota_{\bb{B}})$ \emph{catlike}. 
		 
		We say a functor $F:\CC\to \bb{B}$ between catlike 2-categories is \emph{feline} if it preserves finite strict 2-products, and the diagram 
		\[
		\begin{tikzcd}
			 & \Set \arrow[dr,"\iota_{\bb{B}}"]\arrow[dl,"\iota_{\bb{C}}"']& \\
			\CC \arrow[rr,"F"'] & &\bb{B} 
		\end{tikzcd}
		\]
		commutes up to (necessarily strict) natural isomorphism $\mu$. We will call $(F,\mu)$  a \emph{feline functor}.  
	\end{defn}

	\begin{rmk}
		The notion of a catlike category is not a particularly meaningful notion, but rather an ad hoc axiomatization of the properties of $\Cat$ which are important in our setting. 
	\end{rmk}

	\begin{thm}\label{thm:Omon_func_in_feline}
		Let $\bb{C}$ and $\bb{B}$ be catlike 2-categories.
		\begin{enumerate}
			\item If $(Q,\alpha):\bb{B}\to \bb{C}$ is a feline functor then $Q$ induces a functor 
			\[
			\begin{tikzcd}
				\overline{Q}:&[-3em]\scr{O}\sf{Mon}(\bb{B})\arrow[r] & \scr{O}\sf{Mon}(\bb{C}). 
			\end{tikzcd}
			\]
			\item If $(Q,\alpha),(R,\beta):\bb{B}\to \bb{C}$ are two feline functors and $\mu: Q\Rightarrow R$ is a strict natural isomorphism between them such that $(\mu\circ \iota_{\bb{B}})\star \alpha=\beta$, then $\mu$ induces a  strict natural isomorphism $\overline{\mu}:\overline{Q}\Rightarrow\overline{R}$. 
		\end{enumerate}
	\end{thm}

	\begin{proof}
		We give constructions of  $\overline{Q}$ applied to objects, morphisms, and 2-morphism. Given an $\scr{O}$-pseudomonoid $(\C,\gamma,\phi)$ in $\bb{B}$, We define structure maps to be the composites 
		\[
		\begin{tikzcd}
			\O(n)\times Q(\C)^{\underline{n}} \arrow[r,"\alpha","\cong"'] & Q(\O(n))\times Q(\C)^{\underline{n}}\arrow[r,"\cong"] & Q(\O(n)\times \C^{\underline{n}}) \arrow[r,"Q(\gamma_n)"] & Q(\C).
		\end{tikzcd}
		\]
		where the middle isomorphism is that induced by the universal property of the product. The structure 2-morphisms are given by the pasting diagrams 
		\begin{equation}\label{diag:func_Omon}
		\begin{tikzcd}[every matrix/.append style = {font = \scriptsize},every label/.append style = {font = \tiny},column sep=3em]
			\scr{O}(n)\times {\displaystyle\prod_{i\in\underline{n}}}\left(\scr{O}(|f^{-1}(i)|)\times Q( \sf{C}^{ f^{-1}(i)})\right) \arrow[ddd]\arrow[rrr] \arrow[dr,"\alpha","\cong"']&[-8em] &[-8em] & \O(n)\times Q(\C)^{\underline{n}}\arrow[d,"\alpha","\cong"']\\
			& Q(\scr{O}(n))\times {\displaystyle\prod_{i\in\underline{n}}}\left(Q(\scr{O}(|f^{-1}(i)|))\times Q( \sf{C}^{ f^{-1}(i)})\right) \arrow[dr,"\cong"]\arrow[rr]\arrow[dd]& & Q(\O(n))\times Q(\C)^{\underline{n}})\arrow[d,"\cong"]\\
			& & Q\left(\scr{O}(n)\times {\displaystyle\prod_{i\in\underline{n}}}\left(\scr{O}(|f^{-1}(i)|)\times \sf{C}^{ f^{-1}(i)}\right)\right)\arrow[r,"Q({\on{id}\times \prod \gamma_{f^{-1}(i)}})",""'{name=U}]\arrow[d,"Q(\mu_f\times \on{id})"'] & Q(\scr{O}(n)\times \sf{C}^{\underline{n}}) \arrow[d,"Q(\gamma_n)"]\\
			 \O(m)\times \Q(\C)^{\underline{m}}\arrow[r,"\alpha","\cong"']& Q(\O(m))\times \Q(\C)^{\underline{m}}\arrow[r,"\cong"']&Q(\scr{O}\left(m\right)\times  \sf{C}^{\underline{m}}) \arrow[r,"Q({\gamma_{m}})"',""{name=R}] & Q(\sf{C}) \arrow[from=U,to=R, Leftarrow,shorten >=3ex, shorten <=3ex,"Q(\phi_{f})"] 
		\end{tikzcd}
		\end{equation}
		Where the middle isomorphisms are again the universal maps associated to products. Each of the coherence conditions follows by simply stacking two additional `shells' on the corresponding diagram, one for the universal property of the product, and one for $\alpha$. 
		
		Given a lax $\O$-morphism $(F,\xi)$ from $(\sf{C},\gamma,\phi)$ to $(\sf{D},\lambda,\psi)$ in $\bb{B}$, the corresponding lax $\O$-morphism from $\overline{Q}(\sf{C},\gamma,\phi)$ to $\overline{Q}(\sf{D},\lambda,\psi)$ has underlying functor given by $Q(F)$. The structure 2-morphisms are given by the pasting diagrams 
		\[
		\begin{tikzcd}
			\O(n)\times Q(\C)^n\arrow[d,"\id\times Q(F)"'] \arrow[r,"\alpha\times\id"] & Q(\O(n))\times Q(\C)^n \arrow[r,"\cong"]\arrow[d] & Q(\O(n)\times \C^n) \arrow[r,"Q(\gamma_n)"]\arrow[d] & Q(\C)\arrow[d,"Q(F)"]\\
			\O(n)\times Q(\D)^n \arrow[r,"\alpha\times\id"] & Q(\O(n))\times Q(\D)^n \arrow[r,"\cong"] & Q(\O(n)\times \D^n) \arrow[r,"Q(\lambda)"]\arrow[ur,Rightarrow,"Q(\xi_n)"] & Q(\D)
		\end{tikzcd}
		\]
		It is immediate from the interchange law that this assignment is functorial and unital. The coherence laws follow as the did on objects. 
		
		Finally, let $(\sf{C},\gamma,\phi)$ and $(\sf{D},\lambda,\psi)$ be two $\scr{O}$-pseudomonoids and $(F,\xi),(G,\zeta):\sf{C}\to \sf{D}$ two lax  $\scr{O}$-morphisms in $\bb{B}$, and let $\mu:F\Rightarrow G$  be an $\O$-2-morphism. Then $Q(\mu)$ is an $\O$-2-morphism between $\overline{Q}(F,\xi)$ and $\overline{Q}(G,\zeta)$, by the same style of argument described above. Unitality and functoriality follow from the selfsame properties of $Q$, completing the proof of (1). 
		
		To prove (2), let $\mu$ be a strict natural isomorphism as in the statement of the theorem. We first note that the component $\mu_{(\C,\gamma,\phi)}:=\mu_\C$ defines a strict morphism of $\O$-monoids. Applying the appropriate products of $\mu_\C$ objectwise  to diagram (\ref{diag:func_Omon}),  strict 2-naturality shows that the cube formed from the bottom right square commutes, the universal property of the product shows that the middle shell commutes, and the compatibility with $\alpha$ and $\beta$ ensures that the outer shell commutes. The compatibilities with 1-and 2-morphisms are similarly immediate. 
	\end{proof}

	\begin{cor}\label{cor:2-equiv=equivOmon}
		If $\bb{B}$ and $\bb{C}$ are catlike 2-categories which are  strictly 2-equivalent via feline functors and compatible natural transformations, there is a strict 2-equivalence 
		\[
		\scr{O}\sf{Mon}(\bb{B})\simeq \scr{O}\sf{Mon}(\bb{C}). 
		\]
	\end{cor}

	\begin{proof}
		Since, clearly, $\overline{(\id,\id)}=\id$, it will be sufficient for us to show that, for composable feline functors $(Q,\alpha)$ and $(R,\beta)$, $\overline{(R\circ Q,(R\circ \alpha)\star \beta)}$ is equal to $\overline{(R,\beta)}\circ \overline{(Q,\alpha)}$. 
		
		Applying both functors to $(\C,\gamma,\phi)$, the commutative diagram 
		\[
		\begin{tikzcd}
			R(\O(n)\times Q(\C)^{\underline{n}}) \arrow[r, "R(\alpha\times \id)"] & R(Q(\O(n))\times Q(\C)^{\underline{n}}) \arrow[r,"\cong"] & R(Q(\O(n)\times \C^{\underline{n}})) \arrow[r,"R(Q(\gamma_n))"] & R(Q(\C)) \\
			R(\O(n))\times R(Q(\C))^{\underline{n}}\arrow[u,"\cong"]\arrow[r,"R(\alpha)\times \id"] & R(Q(\O(n)))\times R(Q(\C))^{\underline{n}} \arrow[ur,"\cong"'] \arrow[u,"\cong"] & &\\
			\O(n)\times R(Q(\C))^{\underline{n}}\arrow[u,"\beta\times\id"] \arrow[ur,"(R(\alpha)\circ \beta)\times \id"'] & & & 
		\end{tikzcd}
		\]
		we see that the structure map of $\overline{R\circ Q}(\C,\gamma,\phi)$ --- the bottom composite --- is equal to the structure map of $\overline{R}\circ \overline{Q}(\C,\gamma,\phi)$ --- the top left composite. Applying the same basic construction to the diagram (\ref{diag:func_Omon}) then shows that, in fact, the two functors agree identically on objects. 
		
		Applying both functors to a lax $\O$-morphism $(F,\xi)$ shows that both images are given by the pair $R(Q(F))$ together with the 2-morphism 
		\[
		\begin{tikzcd}
			\O(n)\times R(Q(\C))^n\arrow[d,"\id\times R(Q(F))"'] \arrow[r]  & R(Q(\O(n)\times \C^n)) \arrow[r,"Q(\gamma_n)"]\arrow[d] & R(Q(\C))\arrow[d,"R(Q(F))"]\\
			\O(n)\times R(Q(\D))^n \arrow[r] & R(Q(\O(n)\times \D^n)) \arrow[r,"R(Q(\lambda))"]\arrow[ur,Rightarrow,"R(Q(\xi_n))"] & R(Q(\D))
		\end{tikzcd}
		\]  
		Where the two unlabeled morphisms are given by either the upper left composite in the previous diagram (for $\overline{R}\circ \overline{Q}$) or the lower composite (for $\overline{R\circ Q}$). Since we have already concluded these are equal, this shows that the functors agree on 1-morphisms. It is immediate from construction that they agree on 2-morphisms, and so the proof is complete. 
	\end{proof}

	\section{$\scr{O}$-pseudomonoids in $\Cat$}\label{sec:opsCat}
	
		\begin{defn}
		We call an $\scr{O}$-pseudomonoid in $\Cat$ a \emph{$\scr{O}$-monoidal category}, and denote the category $\scr{O}\sf{Mon}(\Cat)$ simply by $\scr{O}(\Cat)$. We will call the 2-morphisms in this category \emph{lax $\scr{O}$-monoidal functors} and \emph{$\scr{O}$-monoidal transformations} respectively.
	\end{defn}
	
	Since speaking of \emph{monoidal} structures invokes a certain mode of presentation of such structures, let us briefly unpack what it means for a category $\C$ to have a monoidal structure indexed by an operad $\O$. There are two pieces of data:
	\begin{itemize}
		\item For each $n$-ary operation $p \in \O(n)$ a tensoring operation indexed by $p$, which we write as
		\[
		\begin{tikzcd}
			\bigotimes_p : &[-3em] \C^\un \arrow[r] & \C
		\end{tikzcd}
		\]
		Hence, for an ordered $n$-tuple of objects $(A_i)_{i \in \un}$ we have a $p$-indexed tensor product $\bigotimes_p (A_i)_{i \in \un}$.
		
		\item Given a function $f : \um \to \un$ in $\Fin$, an operation $p \in \O(n)$ and operations $q_i \in \O(|f^{-1}(i)|)$, $i \in \un$, we have a structure natural isomorphism 
		\[
		\begin{tikzcd}
			\phi_f :&[-3em] \displaystyle \bigotimes_{p \circ (q_i)_{i \in \un}} (A_j)_{j \in \um} \arrow[r] & \displaystyle \bigotimes_{p} \left( \bigotimes_{q_i} (A_j)_{f(j) = i} \right)_{i \in \un}
		\end{tikzcd}
		\]
		for all $m$-tuples of objects $(A_j)_{j \in \um}$ in $\C$. The notation $\phi_f$ for the above isomorphism is an abuse which leaves as understood in context the fact that $\phi_f$ depends on the $A_i$'s.
	\end{itemize}
	The above data are subject to the following coherence laws:
	\begin{itemize}
		\item Given a composable pair of functions $\ul \xrightarrow{g} \um \xrightarrow{f} \un$ in $\Fin$, an operation $p \in \O(n)$, a sequence of operations $q_i \in \O(|f^{-1}(i)|)$, $i \in \un$, and a sequence of operations $r_j \in \O(|g^{-1}(j)|)$, $j \in \um$, the following square commutes
		\begin{equation}\label{diag:moncat_assoc}
		\begin{tikzcd}[column sep=1em]
			{\displaystyle \bigotimes_{p \circ (q_i)_{i \in \un} \circ (r_j)_{j \in \um}} (A_k)_{k \in \ul}} &&& {\displaystyle \bigotimes_{p} \left( \bigotimes_{q_i \circ (r_j)_{f(j)= i}} (A_k)_{fg(k)=i} \right)_{i \in \un}} \\
			\\
			{\displaystyle \bigotimes_{p \circ (q_i)_{i \in \un}} \left( \bigotimes_{r_j} (A_k)_{g(k) = j} \right)_{j \in \um}} &&& {\displaystyle \bigotimes_p \left( \bigotimes_{q_i} \left( \bigotimes_{r_j} (A_k)_{g(k) = j} \right)_{f(j) = i} \right)_{i \in \un}}
			\arrow["{\phi_g}"', from=1-1, to=3-1]
			\arrow["{\phi_{fg}}", from=1-1, to=1-4]
			\arrow["{\bigotimes_p (\phi_{g|_{f^{-1}(i)}})_{i \in \un}}", from=1-4, to=3-4]
			\arrow["{\phi_f}", from=3-1, to=3-4]
		\end{tikzcd}
		\end{equation}
		for each $l$-tuple of objects $(A_k)_{k \in \ul}$ in $\C$.
		
		\item The functor $\otimes_\eta$ is the identity on $\C$, where $\eta \in \O(1)$ is the identity operation in $\O$. Moreover, for all $p \in \O(n)$, the following structure isomorphisms are identities:
		\[
		\begin{tikzcd}
			\id = \phi_{\id_{\un}} : &[-3em] \bigotimes_p (A_i)_{i \in \un} \arrow[r] & \bigotimes_p (\bigotimes_\eta (A_i))_{i \in \un} 
		\end{tikzcd}
		\]
		\[
		\begin{tikzcd}
			\id = \phi_{t_n} : &[-3em] \bigotimes_p (A_i)_{i \in \un} \arrow[r] & \bigotimes_\eta (\bigotimes_p (A_i)_{i \in \un})
		\end{tikzcd}
		\]
		for all $n$-tuples of objects $(A_i)_{i \in \un}$ in $\C$. In the second equality $t_n : \un \to \underline{1}$ is the terminal function.
	\end{itemize}

	Similarly, we say that an $\O$-\emph{monoidal functor} $F : (\C, \phi) \to (\D, \psi)$ is a functor which is equipped with structure natural isomorphisms 
	\[
	\begin{tikzcd}
		\xi_p : &[-3em] \bigotimes_p F(A_i)_{i \in \un} \arrow[r] & F(\bigotimes_p (A_i)_{i \in \un})
	\end{tikzcd}
	\]
	for all $p \in \O(n)$ and $n$-tuple of objects $(A_i)_{i \in \un}$ in $\C$. These isomorphisms have to be coherent, in the sense that for all functions of finite sets $f : \um \to \un$, and choice of an operation $p \in \O(n)$ and sequence of operations $q_i \in \O(|f^{-1}(i)|)$, $i \in \un$, the following diagram commutes 
	\begin{equation}\label{diag:lax_fun_assoc}
	\begin{tikzcd}[every matrix/.append style = {font = \scriptsize},every label/.append style = {font = \tiny},column sep=3em]
		{\displaystyle \bigotimes_{p \circ (q_i)_{i \in \un}} F(A_j)_{j \in \um}} && {\displaystyle F \left(\bigotimes_{p \circ (q_i)_{i \in \un}} (A_j)_{j \in \um}\right)} \\
		&&&& {\displaystyle F\left( \bigotimes_p \left( \bigotimes_{ q_i} (A_j)_{f(j)=i } \right)_{i \in \un} \right)} \\
		{\displaystyle  \bigotimes_p \left( \bigotimes_{ q_i} F(A_j)_{f(j)=i } \right)_{i \in \un} } && {\displaystyle  \bigotimes_p F \left( \bigotimes_{ q_i} (A_j)_{f(j)=i } \right)_{i \in \un} }
		\arrow["{\xi_{p \circ (q_i)_{i \in \un}}}", from=1-1, to=1-3]
		\arrow["{\phi_f}"', from=1-1, to=3-1]
		\arrow["{\otimes_p (F\xi_{q_i})_{i \in \un}}"', from=3-1, to=3-3]
		\arrow["{F(\phi_f)}", from=1-3, to=2-5]
		\arrow["{\xi_p}"', from=3-3, to=2-5]
	\end{tikzcd}
	\end{equation}
	for all $m$-tuples of objects $(A_j)_{j \in \um}$ in $\C$. Moreover $\xi_1$ must be an identity morphism.

	Lastly, an $\O$-\emph{monoidal natural transformation} $\mu : (F, \xi) \Rightarrow (G, \lambda)$ between $\O$-monoidal functors $(F, \xi), (G, \lambda) : (\C, \phi) \to (\D, \psi)$ consists of a natural transformation such that for any $n$-ary operation $p \in \O(n)$ the following diagram commutes
	\begin{equation}\label{diag:montrans_axiom}
	\begin{tikzcd}
		{\displaystyle \bigotimes_p F(A_i)_{i \in \un}} && {\displaystyle F \left(\bigotimes_p (A_i)_{i \in \un}\right)} \\
		\\
		{\displaystyle \bigotimes_p G(A_i)_{i \in \un}} && {\displaystyle G \left(\bigotimes_p (A_i)_{i \in \un}\right)}
		\arrow["{\xi_p}", from=1-1, to=1-3]
		\arrow["{\lambda_p}", from=3-1, to=3-3]
		\arrow["{\bigotimes_p (\mu_{A_i})_{i \in \un}}"', from=1-1, to=3-1]
		\arrow["{\mu_{\bigotimes_p ({A_i})_{i \in \un}}}", from=1-3, to=3-3]
	\end{tikzcd}
	\end{equation}
	for all $n$-tuples of objects $(A_i)_{i \in \un}$ in $\C$.

	To compare with the classical notions of monoidal and symmetric monoidal categories, we consider the operads of interest $\O = \Comm$ and $\O = \Assoc$. In the latter case, the non-symmetric variant of the operad will also play a key role, as the classical definition is formulated without symmetries of any sort. To more easily square our notion of $\O$-monoidal categories with classical definitions, we will invoke the \emph{unbiased} versions of the latter, whose structure is more obviously similar to ours. 
	
	\begin{rmk}
		Our definition of unbiased monoidal category is that of \cite[3.1]{Leinster} (with minor modification), and our notion of unbiased symmetric monoidal category is cribbed from \cite[Ch. 1]{DeligneMilne}. To our knowledge, there is no full exposition in the literature of the theory of unbiased symmetric monoidal categories, and so we fill in the gaps in the definition of \cite{DeligneMilne}. In this sense, the two equivalences of Theorem \ref{thm:classical_mon_cats} make different statements. The equivalence between unbiased monoidal categories and $\Assoc$-pseudomonoids establishes (after connecting with other results) $\Assoc$-pseudomonoids agree with the classical notion of monoidal category. On the other hand, the equivalence between unbiased symmetric monoidal categories and $\Comm$-pseudomonoids is nearly tautological, and merely provides another way of thinking about the latter. The fact that, as Deligne and Milne point out in \cite[Prop. 1.5]{DeligneMilne}, every symmetric monoidal structure extends to an unbiased symmetric monoidal structure shows that we can consider classical monoidal categories as $\Comm$-pseudomonoids. 
	\end{rmk}

	The following definition is a modification of \cite[Prop 1.5]{DeligneMilne}.

	\begin{defn}
		An \emph{unbiased symmetric monoidal category} is a category $\C$ equipped with the following two pieces of data:
		\begin{itemize}
			\item For each $n \geq 0$, an $n$-fold tensor product functor 
			\[
			\begin{tikzcd}
				\displaystyle   \bigotimes_n : &[-3em] \C^\un \arrow[r] & \C
			\end{tikzcd}
			\]
			We write $\bigotimes_n (A_i)_{i \in \un}$ for the tensor product of an $n$-tuple of objects in $\C$.
			
			\item For each function $f : \um \to \un$ in $\Fin$, a natural structure isomorphism
			\[
			\begin{tikzcd}
				\alpha_f : &[-3em] \displaystyle \bigotimes_m (A_j)_{j \in \um} \arrow[r] & \displaystyle \bigotimes_n \left( \bigotimes_{|f^{-1}(i)|} (A_j)_{f(j) = i} \right)_{i \in \un}.
			\end{tikzcd}
			\]
		\end{itemize}
		These data are required to satisfy the following coherence conditions. 
		\begin{itemize}
			\item The functor $\otimes_1$ is the identity on $\C$. 
			\item If one replaces $\otimes_p$ by $\otimes_n$ and $\phi_f$ by $\alpha_f$, the diagram (\ref{diag:moncat_assoc}) commutes.
			\item The $\alpha_f$ satisfy that 
			\[
			\begin{tikzcd}
				\id = \alpha_{t_n} : &[-3em] \bigotimes_n (A_i)_{i \in \un} \arrow[r] & \bigotimes_1 (\bigotimes_n (A_i)_{i \in \un})
			\end{tikzcd}
			\]
			\[
			\begin{tikzcd}
				\id = \alpha_{\id_{\un}} : &[-3em] \bigotimes_n (A_i)_{i \in \un} \arrow[r] & \bigotimes_n (\bigotimes_1 (A_i))_{i \in \un} 
			\end{tikzcd}
			\]
		\end{itemize}
		Similarly, replacing $\otimes_p$ by $\otimes_n$ and $\phi_f$ by $\alpha_f$ in the preceeding discussion yields definitions of \emph{lax symmetric monoidal functors} between unbiased symmetric monoidal categories and \emph{symmetric monoidal transformations} between them. We denote by $\sf{SMC}$ the 2-category of unbiased symmetric monoidal categories with these 1- and 2-morphisms. 
	\end{defn}

	\begin{defn}
		A (non-symmetric) \emph{unbiased monoidal category} is defined almost identically to the symmetric variant, with the key exception that the maps $f:\um\to \un$ of finite sets appearing in the definition must be (weakly) monotone with respect to the canonical order on the naturals. Formally, an \emph{unbiased monoidal category} is a category $\C$ equipped with the following two pieces of data. 
		\begin{itemize}
			\item For each $n \geq 0$, an $n$-fold tensor product functor 
			\[
			\begin{tikzcd}
				\displaystyle   \bigotimes_n : &[-3em] \C^\un \arrow[r] & \C
			\end{tikzcd}
			\]
			We write $\bigotimes_n (A_i)_{i \in \un}$ for the tensor product of an $n$-tuple of objects in $\C$.
			
			\item For each weakly monotone map $f : \um \to \un$, a natural structure isomorphism
			\[
			\begin{tikzcd}
				\alpha_f : &[-3em] \displaystyle \bigotimes_m (A_j)_{j \in \um} \arrow[r] & \displaystyle \bigotimes_n \left( \bigotimes_{|f^{-1}(i)|} (A_j)_{f(j) = i} \right)_{i \in \un}.
			\end{tikzcd}
			\]
		\end{itemize}
		These data are required to satisfy the following coherence conditions: 
		\begin{itemize}
			\item The functor $\otimes_1$ is the identity on $\C$.
			\item For any pair of composable monotone maps $\begin{tikzcd}
				\ul \arrow[r,"g"] &\um \arrow[r,"f"] & \un 
			\end{tikzcd}$, the diagram given by replacing $\otimes_p$ with $\otimes_k$ whenever $p\in\O(k)$ and $\phi_f$ with $\alpha_f$ in diagram \ref{diag:moncat_assoc} commutes. 
			\item The $\alpha_f$ satisfy that 
			\[
			\begin{tikzcd}
				\id = \alpha_{t_n} : &[-3em] \bigotimes_n (A_i)_{i \in \un} \arrow[r] & \bigotimes_1 (\bigotimes_n (A_i)_{i \in \un})
			\end{tikzcd}
			\]
			\[
			\begin{tikzcd}
				\id = \alpha_{\id_{\un}} : &[-3em] \bigotimes_n (A_i)_{i \in \un} \arrow[r] & \bigotimes_n (\bigotimes_1 (A_i))_{i \in \un} 
			\end{tikzcd}
			\]
		\end{itemize}
		A \emph{lax monoidal functor} between unbiased monoidal categories $(\C,\otimes,\alpha)$ and $(\D,\boxtimes,\beta)$ consists of a functor $F:\C\to \D$ and natural structure isomorphisms 
		\[
		\begin{tikzcd}
			\xi_n: &[-3em] \otimes_n F(A_i)_{i\in \un} \arrow[r] & F\left(\otimes_n(A_i)_{i\in\un}\right)
		\end{tikzcd}
		\]
		such that the following conditions hold.
		\begin{enumerate}
			\item For any monotone map $f:\um\to \un$, the diagram obtained from (\ref{diag:lax_fun_assoc}) by the same replacement conventions as above commutes. 
			\item The natural isomorphism $\xi_1$ is an identity. 
		\end{enumerate}
		Finally, a monoidal transformation from $(F,\xi)$ to $(G,\zeta)$ is a natural transformation $\mu:F\Rightarrow G$ such that the obvious modification of the diagram (\ref{diag:montrans_axiom}) commutes. We denote by $\sf{MC}$ the 2-category of unbiased monoidal categories, lax monoidal functors, and monoidal transformations between them. 
	\end{defn}

	\begin{rmk}
		The preceding definition alters Definition 3.1.1 of \cite{Leinster} in two ways. The first, purely notational, is that we use order-preserving maps of finite sets to encode what Leinster calls `double sequences' in op. cit. The second, more important, is that we require $\otimes_1=\id_{\C}$, rather than fixing a natural structure isomorphism $\otimes_1\simeq  \id_{\C}$ as Leinster does. In view of existing strictification results for monoidal categories (e.g., \cite[Thm. 3.1.6]{Leinster}), every unbiased monoidal category in his sense is equivalent to one in our sense, and so the corresponding 2-categories are equivalent. 
	\end{rmk}

	\begin{rmk}
		Let $(\C,\otimes,\alpha)$ be an unbiased symmetric monoidal category. We may recover the more classical presentation of a symmetric monoidal category as follows.
		First, for each $n$-tuple of objects, we can write $\bigotimes_n (A_i)_{i \in \un} = A_1 \otimes \dots \otimes A_n$.
		Associativity is recorded by the natural isomorphisms $\alpha_f$ for appropriate choice of $f$. For instance, for the function $f : \underline{4} \to \underline{2}$ which maps $1,2 \mapsto 1$ and $3, 4 \mapsto 2$, and any quadruple of objects in $\C$ we have an isomorphism
		\[
		\begin{tikzcd}
			(A_1 \otimes A_2) \otimes (A_3 \otimes A_4)  \arrow[r, "\cong"] & A_1 \otimes A_2 \otimes A_3 \otimes A_4.
		\end{tikzcd}
		\]
		
		Taking $n = 0$, since we have $\C^{\underline{0}} \cong *$, the functor $\bigotimes_0$ just picks an object  $I \in \C$ which serves as a unit for the monoidal structure. The unital laws are recorded by the natural isomorphisms $\alpha_f$ for appropriate choice of $f$. For instance, given $f : \underline{2} \to \underline{3}$, $1 \mapsto 1, 2 \mapsto 3$, for any couple of objects in $\C$ we have an isomorphism 
		\[
		\begin{tikzcd}
			A_1 \otimes I \otimes A_2 \arrow[r, "\cong"] & A_1 \otimes A_2
		\end{tikzcd}
		\]
		
		Taking $\sigma : \un \xrightarrow{\cong} \un$ to be an isomorphism and given that $\bigotimes_1 = \id_\C$, we recover the symmetry in the monoidal structure induces by $\sigma$ via the natural isomorphism $\alpha_\sigma$. For instance, in case $\sigma : \underline{2} \to \underline{2}$ is the non-identity isomorphisms, for any couple of objects we do obtain 
		\[
		\begin{tikzcd}
			A_1 \otimes A_2 \arrow[r, "\cong"]&  A_2 \otimes A_1
		\end{tikzcd}
		\]
		The beauty of the unbiased definition is that  associators, unitors and symmetries are recorded by the natural isomorphisms $\alpha_f$ for appropriate choice of $f$.  The coherence laws ensure these interact in the desired way.
	\end{rmk}
	
	\begin{thm}\label{thm:classical_mon_cats}
		There are equivalences of 2-categories 
		\[
		\begin{aligned}
			\sf{MC}&\simeq \scr{A}\!\on{ssoc}(\Cat)\\
			\sf{SMC}&\simeq \scr{C}\!\on{omm}(\Cat)
		\end{aligned}
		\]
	\end{thm}
	
	\begin{proof}
		The second equivalence is definitional. Indeed, $\Comm$ is the terminal operad, with $\Comm(n) \cong \ast$ for all $n \geq 0$. This way, the data to specify some $\C \in \Comm(\Cat)$ simply restricts to the data for an unbiased symmetric monoidal structure on the category $\C$. Similarly for monoidal functors and transformations.
		
		Recall from Example \ref{example : comm, assoc} that the associative operad has as $n$-ary operations the  permutations of the set $\un$, in which case we write  $\Assoc(n) = \on{Aut}_{\Fin}(\un)$. We have a forgetful 2-functor 
		\[
		\begin{tikzcd}
			U : &[-3em] \Assoc(\Cat) \arrow[r] & \sf{MC}
		\end{tikzcd}
		\]
		defined as follows. Given $(\C,\gamma, \phi) \in \Assoc(\Cat)$, the monoidal category $U(\C,\gamma, \phi)$ has as underlying category $\C$, as $n$-fold tensor product the functor $\bigotimes_n = \gamma_{\id_\un} : \C^\un \to \C$ and structure isomorphisms given by $\phi_f$ for order preserving functions $f : \um \to \un$. Similarly, given a functor $(F, \xi) : (\C,\gamma, \phi) \to (\D,\lambda, \psi)$ in $\Assoc(\Cat)$, we define the monoidal functor $UF : U(\C,\gamma, \phi) \to U(\D,\lambda, \psi)$ to have $F$ as underlying functor with structure isomorphisms given by $\xi_{\id_\un}$. On the level of transformations the data remains the same. It is immediate that this construction is 2-functorial, and so we turn to demonstrating that $U$ is an equivalence of 2-categories. 
		
		First, we see that $U$ is essentially surjective as a 2-functor. Let $(\C , \alpha)$ be a monoidal category. We may extend the data for the monoidal structure to the data for an $\Assoc$-monoidal structure $(\C, \phi)$ (with the same underlying category $\C$) as follows. For each permutation $p : \un \xrightarrow{\cong} \un$ and $n$-tuple of objects $(A_i)_{i \in \un}$ we define\footnote{For clarification of how $p^{-1}(i)$ makes its appearance, see Remark \ref{rmk:assoc_alg_set}.}
		\[
		\displaystyle \bigotimes_p (A_i)_{i \in \un} = \bigotimes_n(A_{p^{-1}(i)})_{i\in \un}.
		\]
		Moreover, given a monotone map $f : \um \to \un$ we define $\phi_f = \alpha_f$. In case $f : \un \to \un$ is an isomorphism, let $\phi_f = \id$. For generic $f$, factor $f=g\circ h$ where $g$ is a monotone map and $h$ is a permutation. Notice that, while there may be several different permutations $h$ which fit in to factorizations of this kind, $g$ is uniquely determined by it. Thus, the requirement that diagram (\ref{diag:moncat_assoc}) commute in this case determines $\phi_f=\phi_{gh}$ in terms of composites of the $\phi_h$ and $\phi_g$ (or the inverses of these). This is independent of the choice of factorization precisely because $\phi_h=\id$ for \emph{any} permutation $h$. In particular, regardless of the choice of factorization, the vertical morphisms in (\ref{diag:moncat_assoc}) collapse to identities. 
		
		 To see that associativity is then satisfied in general, let $\begin{tikzcd}
			\ul \arrow[r,"g"] &\um \arrow[r,"f"] & \un 
		\end{tikzcd}$ be composable morphisms in $\Fin$, $\sigma\in \Assoc(n)$, $\tau_{i}\in \Assoc(f^{-1}(i))$, and $\xi_{j}\in \Assoc(g^{-1}(j))$. By factoring $f=\overline{f}\circ p$ as an order-preserving map and a permutation $p$, and then factoring $p\circ g=\overline{g}\circ q$ as an order-preserving map and a permutation $q$, we reduce the check of associativity to checking associativity for the maps $\overline{f}$ and $\overline{g}$, using the operations $\sigma$, $\{p\circ \tau_i\circ p^{-1}\}$, and $\{q\circ \xi_j \circ q^{-1}\}$. Thus, associativity follows from associativity for order-preserving maps. 
		
		The unitality laws are clearly satisfied. Moreover, by construction, we have $U(\C, \phi) = (\C, \alpha)$, and hence $U$ is surjective on objects.
		
		Next, let us see that for any two objects $(C, \phi)$ and $(\D, \psi)$ in $\Assoc(\Cat)$ the induced functor 
		\[
		\begin{tikzcd}
			U : &[-3em] \Assoc(\Cat)((\C,\gamma, \phi), (\D,\lambda, \psi)) \arrow[r] & \sf{MC}(U(\C,\gamma, \phi), U(\D,\lambda, \psi))
		\end{tikzcd}
		\]
		(which we still denote by $U$)
		is an equivalence of categories. We first demonstrate essential surjectivity on 1-morphisms.

		Let $(\C,\gamma,\phi)$ and $(\D,\lambda, \psi)$ be $\Assoc$-monoidal categories, and let $(F, \zeta) : U(\C,\gamma,\phi) \to (\D,\lambda, \psi)$ be a lax monoidal functor. Then for a permutation $f:\un\to \un$, the coherence condition requires that the diagram
		\[
		\begin{tikzcd}[every matrix/.append style = {font = \scriptsize},every label/.append style = {font = \tiny},column sep=2em]
			{\displaystyle \bigotimes_{f} F(A_j)_{j \in \un}}\arrow[rr,red,"\zeta_f"] \arrow[dd,"\phi_f"']&& {\displaystyle F \left(\bigotimes_{f} (A_j)_{j \in \un}\right)}\arrow[drr,"F(\phi_f)","\cong"'] \\
			&&&& {\displaystyle F\left( \bigotimes_{\id_{\un}} \left( \bigotimes_1 (A_{f^{-1}}) \right)_{i \in \un} \right)} \\
			{\displaystyle  \bigotimes_{\id_{\un}} \left( \bigotimes_{1} F(A_{f^{-1}(i)}) \right)_{i \in \un} }\arrow[rr,"\otimes_{\id_{\un}} (F\zeta_1)_{i\in\un}"] && {\displaystyle  \bigotimes_{\id_{\un}} F \left( \bigotimes_{1} (A_{f^{-1}(i)}) \right)_{i \in \un} }\arrow[urr,"\zeta_n"']
		\end{tikzcd}
		\] 
		commutes, uniquely specifying $\zeta_f$. The coherence diagrams in general then follow from the conditions for $\zeta_{n}$ by factoring morphisms in $\Fin$ as before. It follows that $U$ is bijective on objects of hom-categories. 
		
		Finally, we note that, given an monoidal transformation $\mu:U(F,\zeta)\to U(G,\lambda)$, the morphism $\mu$ commutes with the $\zeta_{\id_{\un}}$ and $\lambda_{\id_{\un}}$. Moreover, by naturality, $\mu$ commutes with the instances of $\phi_f$ in the pentagon above. Thus, $\mu$ commutes with the $\zeta_{\id_{\un}}$ and $\xi_{\id_{\un}}$ if and only if it commutes with $\zeta$ and $\xi$ generally. Thus, $U$ is bijective on 2-morphisms, and the theorem is proven. 
	\end{proof}

 	\section{$\scr{O}$-pseudomonoids in $\DFib$}\label{sec:opsDFib}
 	
 	Having now established the connection between $\scr{O}$-monoidal categories and classical notions of monoidal category, we turn our attention to the study of $\O$-pseudomonoids in $\DFib$. 
 	
 	We will denote an object in $\DFib$ generically by $p:\C_1\to \C_0$, a morphism by its pair of constituent morphisms $f=(f^1,f^0)$ and a 2-morphism by its pair of constituent 2-morphisms $\nu=(\nu^1,\nu^0)$. Note that there the functors 
 	\[
 	\begin{tikzcd}
 		\on{source}: &[-3em] \DFib\arrow[r] & \Cat
 	\end{tikzcd}
 	\]
 	and 
 	\[
 	\begin{tikzcd}
 		\on{target}: &[-3em] \DFib\arrow[r] & \Cat
 	\end{tikzcd}
 	\]
 	which take the total category and base category, respectively, are feline, and, indeed, \emph{strictly} respect the inclusion functors from $\Set$. 
 	
 	\begin{defn}
 		Let $\scr{O}\sf{Fib}$ be the 2-category whose objects are \emph{strict} $\scr{O}$-monoidal functors $p:(\C_1,\gamma^1,\phi^1) \to (\C_0,\gamma^2,\phi^2)$ between $\scr{O}$-monoidal categories whose underlying functors are discrete fibrations. The 1-morphisms of $\scr{O}\sf{Fib}$  are (strictly) commutative squares of lax $\scr{O}$-functors (in the sense that the underlying functors strictly commute with $p$, and the structure 2-morphisms strictly commute with $p$), and the 2-morphisms are pairs of $\scr{O}$-monoidal transformations whose underlying transformations form a 2-morphism of discrete fibrations.
 	\end{defn}
 	
 	\begin{thm}\label{thm:ODFib}
 		There is a strict equivalence of 2-categories 
 		\[
 		\begin{tikzcd}
 			\scr{O}\sf{Mon}(\DFib)\arrow[r] & \scr{O}\sf{Fib}
 		\end{tikzcd}
 		\]
 		which, on objects, sends $(p:\C_1\to\C_0,\gamma, \phi)$ to $(\C_1,\gamma^1,\phi^1)$ and $(\C_0,\gamma^0,\phi^0)$. 
 	\end{thm}
 
 	\begin{proof}
 		Since the source and target functors are feline, Theorem \ref{thm:Omon_func_in_feline} guarantees that 
 		\[
 		\overline{\on{source}}(p:\C_1\to\C_0,\gamma, \phi)=(\C_1,\gamma^1,\phi^1)
 		\]
 		and 
 		\[
 		\overline{\on{target}}(p:\C_1\to\C_0,\gamma, \phi)=(\C_0,\gamma^0,\phi^0)
 		\]
 		are $\scr{O}$-monoidal categories. Moreover, the associativity and unitality conditions only involve conditions in either the source or the target, but not both, and so only remaining condition is the compatibility of the $\phi^i$ with the map $p$. This takes the form of the commutative cube
 		\[
 		\begin{tikzcd}[column sep=2em,row sep=3em,every matrix/.append style = {font = \scriptsize},every label/.append style = {font = \tiny}]
 			\scr{O}(n)\times {\displaystyle \prod_{i\in \underline{n}}}\left(\scr{O}(f^{-1}(i))\times \C_1^{f^{-1}(i)}\right) \arrow[rr,"\id\times\Pi_i(\id \times p^{f^{-1}(i)})"]\arrow[dd,"(\mu_f\times \id)\circ \tau"']\arrow[dr,"\id\times\Pi_i\gamma^1_{f_(i)}"] & & \scr{O}(n)\times {\displaystyle \prod_{i\in \underline{n}}}\left(\scr{O}(f^{-1}(i))\times \C_0^{f^{-1}(i)}\right)\arrow[dd,"(\mu_f\times \id)\circ \tau"'{pos=0.6}] \arrow[dr,"\id\times \prod_i \gamma^0_{f_i}"] & \\
 			& { \scr{O}(n)\times \C_1^{\underline{n}}}\arrow[rr,"\id\times p^{\underline{n}}"{pos=0.7}]\arrow[dd,"\gamma^1_n"{pos=0.6}]& & \scr{O}(n)\times \C_0^{\underline{n}}\arrow[dd,"\gamma^0_n"]\\
 			\scr{O}(m) \times C^{\underline{m}}\arrow[rr,"\id\times p"{pos=0.6}]\arrow[dr,"\gamma^1_m"']\arrow[ur,Rightarrow,"\phi^1_f"] & & {\scr{O}(m)\times \C_0^{\underline{m}}}\arrow[dr,"\gamma^0_m"] \arrow[ur,Rightarrow,"\phi^0_f"]\\ 
 			& \C_1  \arrow[rr,"p"'] & & \C_0 
 		\end{tikzcd}
 		\] 
 		and so is equivalent to $p$ being a strict $\O$-monoidal functor. Since this is simply a reformulation of conditions on the same data, this shows that the assignment is essentially surjective. 
 		
 		Now, given a lax $\scr{O}$-morphism $((F^1,F^0),(\xi^1,\xi^0))$ from $(p:\C_1\to \C_0, \gamma,\phi)$ to $(q:\D_1\to \D_0, \lambda,\psi)$ in $\O\sf{Mon}(\DFib)$, Theorem \ref{thm:Omon_func_in_feline} applied to the source and target functors shows that $(F^1,\xi^1)$ and $(F^0,\xi^0)$ define lax $\O$-functors between the corresponding $\O$-monoidal categories. Once again, the associativity and unitality conditions for this lax $\O$-morphism only put conditions on the components, and so the only condition which remains for us to unwind is compatibility with $p$. However, this is precisely the condition that both $(F^1,F^0)$ and $(\xi^1_n,\xi^0_n)$ strictly commute with $p$. Thus, such a lax $\O$-morphism $((F^1,F^0),(\xi^1,\xi^0))$ is the same thing as a 1-morphism $((F^1,\xi^1),(F^0,\xi^0))$ in $\O\sf{Fib}$. 
 		
 		Finally, suppose that $(\nu^1,\nu^0)$ is an $\O$-2-morphism between $((F^1,F^0),(\xi^1,\xi^0))$ and $((G^1,G^0),(\zeta^1,\zeta^0))$ in $\O\sf{Mon}(\DFib)$. As before, the fact that the source and target functors are feline implies that $\nu^1$ and $\nu^0$ are $\O$-monoidal transformations between their respective $\O$-monoidal functors. Also as before, the conditions that $(\nu^1,\nu^0)$ must satisfy to be an $\O$-2-morphism only put the selfsame conditions on the components $\nu^1$ and $\nu^2$, and so we need only note that, by definition $(\nu^1,\nu^0)$ is also a 2-morphism of discrete fibrations. 
 		
 		Functoriality and strict fully faithfulness follow immediately from the construction, together with the functoriality of $\overline{\on{source}}$ and $\overline{\on{target}}$, completing the proof. 
 	\end{proof}

	\section{$\scr{O}$-pseudomonoids in $\ISet$}\label{sec:opsISet}
	
	Let us fix an operad $\scr{O}$. The commutative operad $\scr{C}\!\on{omm}$ is terminal, and thus pulling back the Cartesian symmetric monoidal structure on $\Set$ along the unique map $\scr{O}\to \scr{C}\!\on{omm}$ yields a canonical $\scr{O}$-monoidal structure on $\Set$. We will exclusively consider this $\scr{O}$-monoidal structure on $\Set$ in this section.  
	
	Before we got further in identifying $\scr{O}$-pseudomonoids in $\ISet$, let us understand the structure maps of this $\scr{O}$-monoidal structure. We will denote by
	\[
	\begin{tikzcd}
		p_n :&[-3em] \Set^{\underline{n}}\arrow[r] & \Set 
	\end{tikzcd}
	\]
	the limit functors. For simplicity in the following proof, we will take the structure maps of the Cartesian symmetric monoidal structure on $\Set$ for $n\neq 1$ to be the maps 
	\[
	\begin{tikzcd}
		\pi_n:&[-3em] \ast \times \Set^{\underline{n}} \arrow[r,"\ast_{\ast}\times \id"] & \Set^{\underline{n+1}} \arrow[r,"p_n"] & \Set,  
	\end{tikzcd}
	\]
	where, for $X\in \Set$, $\ast_X:X\to \Set$ is the constant functor on a chosen terminal object. The structure transformations $\alpha_f$ are then the unique natural isomorphisms between composites of these functors determined by the universal properties of limits.

	\begin{defn}
		Define $\ISet^{\scr{O},\on{lax}}$ to be the 2-category whose objects are lax $\scr{O}$-monoidal functors to $\Set$, whose morphisms are diagrams 
		\[
		\begin{tikzcd}
			\I\arrow[dr,"F",""'{name=U}]\arrow[dd,"M"'] & \\
			& \Set \\
			\J\arrow[ur,"G"'] & \arrow[from=U,to=3-1,Rightarrow,"\mu",shorten >= 0.5em]
		\end{tikzcd}	
		\]
		of lax $\scr{O}$-monoidal functors and $\scr{O}$-monoidal transformations, and whose 2-morphisms are $\scr{O}$-monoidal transformations. 
	\end{defn}

	\begin{thm}\label{thm:OISet}
		There is a strict equivalence of 2-categories
		\[
		\scr{O}\sf{Mon}(\ISet)\simeq \ISet^{\scr{O},\on{lax}}.
		\]
	\end{thm}

	\begin{proof} 
			We first construct a 2-functor 
		\[
		\begin{tikzcd}
			\Phi:&[-3em] \scr{O}\sf{Mon}(\ISet) \arrow[r] & \ISet^{\scr{O},\on{lax}}
		\end{tikzcd}
		\]
		as follows. Given a $\scr{O}$-monoid $(F:\I\to \Set, \gamma,\phi)$ in $\ISet$, each morphism $\gamma_n$ takes the form of a functor, $G_n:\scr{O}(n)\times\I^{\underline{n}}\to \I$ and a natural transformation $\nu_n:\scr{O}(n)\times F^{\underline{n}}\Rightarrow F\circ G_n$. The 2-morphisms $\phi_f$ are then simply natural transformations relating the $G_n$, which are in addition compatible with the $\nu_n$. Forgetting the data of the $F^n$'s and $\nu_n$ then immediately yields a $\scr{O}$-monoidal structure on $\I$, with structure maps $G_n$ and structure 2-morphisms $\phi_f$. 
		
		We then note that, by definition of $F^{\underline{n}}$, we can reformulate the fact that $\nu_n$ are transformations from $F^{\underline{n}}\Rightarrow F\circ G_n$ as the fact that there are diagrams  
		\[
		\begin{tikzcd}
			\scr{O}(n)\times \I^{\underline{n}}\arrow[d,"G_n"'] \arrow[r,"\id\times F^n"] & \scr{O}(n)\times \Set^{\underline{n}}\arrow[dl,Rightarrow,"\nu"]\arrow[r] \arrow[dr] & \Set^{\underline{n+1}} \arrow[d,"p_n"]\\
			\I \arrow[rr,"F"'] &  &\Set 
		\end{tikzcd}
		\]
		Thus, $\nu_n$ is equivalently a transformation filling the diagram 
		\[
		\begin{tikzcd}
			\scr{O}(n)\times \I^{\underline{n}}\arrow[r,"G_n"]\arrow[d,"\id\times F^n"'] & \I\arrow[d,"F"] \\
			\scr{O}(n)\times \Set^{\underline{n}}\arrow[r,"\pi_n"] \arrow[ur,Rightarrow,"\nu_n"] & \Set 
		\end{tikzcd}
		\]
		Writing out the compatibility of the $\nu_n$ with the $\phi_f$ yields the associativity cube for a lax $\scr{O}$-morphism. To see this we note that the necessary compatibility is the commutativity of the pyramid 
		\[\begin{tikzpicture}[z={(10:10mm)},x={(-45:10mm)},scale=0.9]
			\def\xstep{8}
			\def\ystep{5}
			\begin{scope}[canvas is zy plane at x=-2]
				\node[transform shape] at (0,\ystep) (OOF){\scriptsize$\O(n)\times \prod_{i\in\underline{n}}\left(\O(f^{-1}(i))\times \I^{f^{-1}(i)}\right)$};
				\node[transform shape] at (\xstep,\ystep)(OnI) {\scriptsize$\O(n)\times \I^{\underline{n}}$};
				\node[transform shape] at (0,0) (OmI) {\scriptsize$\O(m)\times \I^{\underline{m}}$};
				\node[transform shape] at (\xstep,0)(I) {\scriptsize$I$};
				
				\draw[->,transform shape] (OOF) -- node[above]{\tiny$\id\times \Pi_i(G_{f^{-1}(i)})$} (OnI);
				\draw[->,transform shape] (OOF) -- node[left]{\tiny$(\mu_f\times\id)\circ\tau_f$} (OmI);
				\draw[->,transform shape] (OnI) -- node[right] {\tiny$\gamma_n$} (I);
				\draw[->,transform shape] (OmI) -- node[below] {\tiny$\gamma_m$} (I);
				\draw[-{Implies},double distance=1.5pt,transform shape] (OmI) to node[above left] {\scriptsize$\phi_f$} (OnI);
			\end{scope}
			
			\begin{scope}[canvas is zy plane at x=2]
				\node at (4,2.5) (set) {\scriptsize$\Set$};
			\end{scope}
			\draw[->] (OOF) -- (set);
			
			\draw[->] (OmI) -- (set);
			\draw[->] (I) -- (set);
			\draw[->] (OnI) -- (set);
			\path (10.5,7.5) node{};
		\end{tikzpicture}\]
		Examining the construction, we see that we can expand out the faces of this pyramid as 
		\begin{itemize}
			\item \textsc{Left Face:}
			\[
			\begin{tikzcd}[every matrix/.append style = {font = \scriptsize},every label/.append style = {font = \tiny},column sep=10em, row sep=4em]
				\O(n)\times \prod_{i\in\underline{n}}\left(\O(f^{-1}(i))\times \I^{f^{-1}(i)}\right) \arrow[r]\arrow[d,"(\mu_f\times \id)\circ \tau"'] &[-7em] \O(n)\times \prod_{i\in\underline{n}}\left(\O(f^{-1}(i))\times \Set^{f^{-1}(i)}\right)\arrow[r,"\pi_{n}\circ (\id\times \Pi_i\pi_{f^{-1}(i)})"]\arrow[d,"(\mu_f\times \id)\circ \tau"'] & \Set \arrow[d,"\id"] \\
				\O(m)\times \I^{\underline{m}} \arrow[r] & \O(m)\times \Set^{\underline{m}} \arrow[r,"\pi_{m}"'] & \Set 
			\end{tikzcd}
			\] 
			where the left-hand horizontal morphisms are the appropriate products of identities and powers of $F$. The left-hand square commutes strictly, and the right-hand square is filled by a natural isomorphism uniquely determined by the universal properties of products. 
			\item \textsc{Bottom Face:}
			\[
			\begin{tikzcd}[row sep=2em,column sep=2.5em]
				\O(m)\times \I^{\underline{m}}\arrow[r,"\id\times F^m"]\arrow[d,"G_m"'] & \O(m)\times \Set^{\underline{m}} \arrow[r,"\pi_{m}"]\arrow[d,"\pi_{m}"]\arrow[dl,Rightarrow,"\nu_m"] & \Set \arrow[d,"\id"] \\
				\I\arrow[r,"F"']& \Set\arrow[r,"\id=\pi_1"'] & \Set 
			\end{tikzcd}
			\]
			\item \textsc{Right Face:} Identical to the bottom face, with $m$ replaced by $m$. 
			\item \textsc{Top Face:}
			\[
			\begin{tikzcd}[every matrix/.append style = {font = \scriptsize},every label/.append style = {font = \tiny},column sep=10em, row sep=4em]
				\O(n)\times \prod_{i\in\underline{n}}\left(\O(f^{-1}(i))\times \I^{f^{-1}(i)}\right) \arrow[r]\arrow[d,"\id\times \Pi_iG_{f^{-1}(i)}"'] &[-7em] \O(n)\times \prod_{i\in\underline{n}}\left(\O(f^{-1}(i))\times \Set^{f^{-1}(i)}\right)\arrow[r,"\pi_{n}\circ (\id\times \Pi_i\pi_{f^{-1}(i)})"]\arrow[d,"\id\times \Pi_i\pi_{f^{-1}(i)}"'] \arrow[dl,Rightarrow,"\id\times\Pi_i\nu_{f^{-1}(i)}"']  & \Set \arrow[d,"\id"] \\
				\O(n)\times \I^{\underline{n}} \arrow[r] & \O(n)\times \Set^{\underline{n}} \arrow[r,"\pi_{n}"'] & \Set 
			\end{tikzcd}
			\]
		\end{itemize}
		We thus see that this condition is formed by juxtaposing the associativity cube for a lax $\O$-morphism with a cube whose sides are exclusively identities natural isomorphisms determined by the universal property of the product. Moreover, the morphism from the terminal vertex of the associativity cube to the terminal vertex of the whole cube is the identity on $\Set$. Thus, compatibility is equivalent to the associativity condition for a lax $\O$-morphism and so we are left to check unitality. 
		
		The conditions (unitality 2) and (unitality 3) are simply conditions guaranteeing that $(\I,G,\phi)$ is an $\scr{O}$-pseudomonoid. Expanding out (unitality 1) yields the diagram 
		
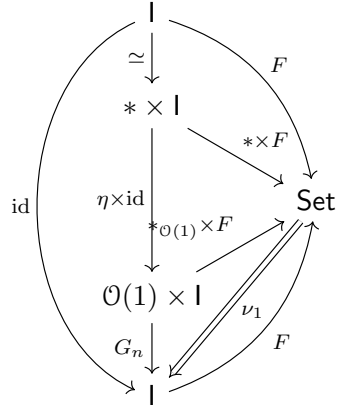
\[
		\begin{tikzcd}
			\I\arrow[d,"\simeq"']\arrow[ddr,bend left,"F"]\arrow[dddd,bend right=6em,"\id"'] & \\
			\ast\times \I\arrow[dd,"\eta\times \id"']\arrow[dr,"\ast\times F"] & \\
			& \Set\arrow[ddl,Rightarrow,"\nu_1"] \\
			\scr{O}(1)\times \I\arrow[d,"G_n"'] \arrow[ur,"\ast_{\scr{O}(1)}\times F"]& \\
			\I \arrow[uur, bend right, "F"']&  
		\end{tikzcd}
		\] 
		and requires that (1) the left-hand bubble commutes and (2) the total pasting diagram equals the identity. The first condition is simply one of the requirements that $(\I,G,\phi)$ be an $\scr{O}$-pseudomonoid. The second condition can be rewritten as 
		\[
		\begin{tikzcd}
			\I\arrow[d,"\simeq"']\arrow[r,"F"] & \Set\arrow[d,"\simeq"]\\		\ast\times \I\arrow[d,"\eta\times \id"']\arrow[r,"\id\times F"] &\ast \times \Set \arrow[d,"\eta\times \id"] \\
			\scr{O}(1)\times \I\arrow[d,"G_n"'] \arrow[r,"\scr{O}(1)\times F"]& \scr{O}(1)\times \Set\arrow[dl,Rightarrow,"\nu_1"]\arrow[d,"\pi_1"]\\
			\I \arrow[r, "F"']&  \Set
		\end{tikzcd}
		\]
		showing that $(F,\nu)$ is precisely an $\scr{O}$-monoidal functor. We thus define $\Phi(F:\I\to\Set,\gamma,\phi)$ to be the $\scr{O}$-monoidal category $(\I,G_n,\phi)$ together with the lax monoidal functor $(F,\nu)$. Note that, as this is simply a reformulation of the data, this automatically defines an essentially surjective assignment on objects. 
		
		Given a lax $\scr{O}$-morphism  $(U,\xi):(F,\gamma,\phi)\to (\widetilde{F},\widetilde{\gamma},\widetilde{\phi})$, we see that $U$ consists of a functor $U:\I\to \widetilde{I}$ and a transformation $\zeta: F\Rightarrow \widetilde{F}\circ U$. Similarly, $\xi$ is simply a collection of transformations relating the $\scr{O}$-monoidal structures on $\I$ and $\widetilde{\I}$ compatible with $\zeta$. We thus see that $(U,\xi)$ defines a lax monoidal functor, and we are left to unwind the compatibility of $\xi$ and $\zeta$.  
		
		This compatibility condition requires the equality of pasting diagrams 
		\[
		\begin{tikzcd}[column sep=3em]
			& \scr{O}(n)\times I^{\underline{n}}\arrow[dl,"G_n"']\arrow[dr,"\ast_{\scr{O}(n)}\times F^{\underline{n}}"] \arrow[dd,"\id\times U^n"]& \\
			I\arrow[dd,"U"'] & & \Set\arrow[dd,"\id"] \arrow[dl, Rightarrow,"\id\times \zeta^n"]\\ 
			& \scr{O}(n)\times \widetilde{I}^{\underline{n}}\arrow[dr,"\ast_{\scr{O}(n)}\times \widetilde{F}^{\underline{n}}",""'{name=U}]\arrow[dl,"\widetilde{G}_n"']\arrow[ul,Rightarrow,"\xi_n"] & \\
			\widetilde{I}\arrow[rr,"\widetilde{F}"'] & & \Set \arrow[from=U,to=4-1,Rightarrow,"\widetilde{\nu}_n"'{pos=0.3},shorten >=2em] 
		\end{tikzcd} \quad = \quad 
			\begin{tikzcd}[column sep=3em]
			& \scr{O}(n)\times \I^{\underline{n}}\arrow[dl,"G_n"']\arrow[dr,"\ast_{\scr{O}(n)}\times F^{\underline{n}}",""'{name=Z}] & \\
			\I\arrow[dd,"U"']\arrow[rr,"F"'] & & \Set\arrow[dd,"\id"]\arrow[ddll,Rightarrow,"\zeta"] \\ 
			&  & \\
			\widetilde{I}\arrow[rr,"\widetilde{F}"'] & & \Set \arrow[from=Z,to=2-1,Rightarrow,shorten >=2em,"\nu_n"'{pos=0.3}]
		\end{tikzcd} 
		\]
		Expanding out morphisms to $\Set$ as composites with $\pi_n:\scr{O}\times \Set^{\underline{n}}\to \Set$, we see that this is equivalently the requirement that the pasting diagrams  
		\[
		\begin{tikzcd}
			\scr{O}(n)\times \I^{\underline{n}}\arrow[r,"G_n"]\arrow[d,"\id\times U^n"']\arrow[dd,bend right=4cm,"\id\times F^n"',""{name=U}] & \I\arrow[d,"U"] \\
			\scr{O}(n)\times \widetilde{\I}^{\underline{n}}\arrow[r,"\widetilde{G}_n"]\arrow[d,"\id\times \widetilde{F}^n"'] \arrow[ur,Rightarrow,"\xi_n"]& \I\arrow[d,"\widetilde{F}"]\\
			\scr{O}(n)\times \Set^{\underline{n}}\arrow[r,"\pi_n"']\arrow[ur,Rightarrow,"\widetilde{\nu}_n"] & \Set \arrow[from=U,to=2-1,Rightarrow,"\zeta^n"]
		\end{tikzcd} \quad = \quad \begin{tikzcd}
			\scr{O}(n)\times \I^{\underline{n}}\arrow[r,"G_n"]\arrow[d,"\id\times F^n"'] & \I\arrow[d,"F"',""{name=F}] \arrow[d,bend left=3cm,"\widetilde{F}\circ U",""'{name=Z}] \\
			\scr{O}(n)\times \Set^{\underline{n}}\arrow[r,"\pi_n"']\arrow[ur,Rightarrow,"\nu_n"] & \Set \arrow[from=F,to=Z,Rightarrow,"\zeta"{pos=0.1}]
		\end{tikzcd}
		\]
		are equal. Thus, this is equivalent to requiring that $\zeta$ be a monoidal natural transformation from $F$ to $\widetilde{F}\circ U$. Since this is simply a reformulation of the data, this establishes that $\Phi$ is a bijection on 1-hom-sets. 
		
		The assignment on 2-morphisms is obvious, and is clearly also bijective on 2-hom-sets, completing the proof.  
	\end{proof}

	\section{The $\O$-monoidal Grothendieck construction}\label{sec:OGroth}

	\begin{prop}
		The 2-functors $\int$ and $\on{T}$ are feline, and the natural transformations $\Phi$ and $\Psi$ of Theorem \ref{thm:classGC} are compatible with their feline structures. 
	\end{prop}

	\begin{proof}
		Since the functors $\int$ and $\on{T}$ form a strict 2-equivalence, preservation of products is immediate. Given a set $X$, we view it as the discrete fibration $\iota_{\DFib}(X)$ given by $\id:X\to X$ and the indexed set $\iota_{\ISet}(X)$ given by the functor $X\to \Set$ which sends each element of $X$ to the terminal set. Since $\on{T}(\iota_{\DFib}(X))$ is also a functor $X\to \Set$ sending every element to a terminal set, there is a unique natural isomorphism $\mu_X$ filling the diagram
		\[
		\begin{tikzcd}
			X\arrow[dr,"\on{T}(\iota_{\DFib}(X))"]\arrow[dd,"\id"'] & \\
			& \Set \\
			X\arrow[ur,"\iota_{\ISet}(X)"'] & 
		\end{tikzcd}
		\]
		It is a simple matter to check that the components $(\id_X,\mu_X)$ define a natural isomorphism $T\circ \iota_{\DFib}\cong \iota_{\ISet}$, showing that $\on{T}$ is feline. 
		
		On the other hand, suppose that we note that $\int\iota_{\ISet}(X)$ is the map $X\times \{\ast\}\to X$ which sends $(x,\ast)\mapsto X$. Since the fibres are all singletons, there is a unique isomorphism 
		\[
		\begin{tikzcd}
			X\times \{\ast\} \arrow[r,"\on{pr}_1"]\arrow[d,"\on{pr}_1"'] & X\arrow[d,"\id"]\\
			X\arrow[r,"\id"] & X
		\end{tikzcd}
		\]
		from $\int\iota_{\ISet}(X)$ to $\iota_{\DFib}(X)$. It is again immediate that the components $(\on{pr}_1^X,\id_X)$ form a natural isomorphism $\int\circ \iota_{\ISet}\cong \iota_{\DFib}$. 
		
		It is an easy exercise to check that $\Phi$ and $\Psi$ are compatible with these feline structures. 
	\end{proof}
	
	\begin{thm}\label{thm:Omon-GC}
		For any operad $\scr{O}$, the classical Grothendieck construction $\int:\ISet\to \DFib$ induces an equivalence of 2-categories 
		\[
		\begin{tikzcd}
			{\displaystyle \int^{\scr{O}}}: &[-3em] \ISet^{\scr{O},\on{lax}}\arrow[r] & \scr{O}\sf{Fib}. 
		\end{tikzcd}
		\]
	\end{thm}

	\begin{proof}
		Applying Corollary \ref{cor:2-equiv=equivOmon} and Theorem \ref{thm:classGC}, we see that $\int$ induces a 2-equivalence 
		\[
		\begin{tikzcd}
			\scr{O}\sf{Mon}(\ISet)\arrow[r] & \scr{O}\sf{Mon}(\DFib).
		\end{tikzcd}
		\]
		However Theorem \ref{thm:ODFib} identifies the target with $\scr{O}\sf{Fib}$ and Theorem \ref{thm:OISet} identifies the source with $\ISet^{\scr{O},\on{lax}}$, completing the proof. 
	\end{proof}

	\begin{cor}
		For any operad $\scr{O}$ and any $\scr{O}$-monoidal category $\I$, the Grothendieck construction induces an equivalence of 1-categories
		\[
		\Fun^{\scr{O},\on{lax}}(\I,\Set)\simeq \scr{O}\sf{Fib}_{\I}.
		\] 
	\end{cor}

	\begin{proof}
		 This is simply restricting each side of Theorem \ref{thm:Omon-GC} to a single $\scr{O}$-monoidal base category $\I$. 
	\end{proof}

	\printbibliography

	\appendix
	
	\section{Additional diagrams}
	
	\begin{equation}\label{diag:assoc_for_opfunct}
		\begin{tikzpicture}[z={(10:10mm)},x={(-45:10mm)}]
			\def\xstep{8}
			\def\ystep{5}
			\begin{scope}[canvas is zy plane at x=-2]
				\node[transform shape] at (0,\ystep) (PBTC){\scriptsize$\scr{P}(n)\times \prod_{i\in\underline{n}}\left(\scr{P}(g^{-1}(i))\times \C^{g^{-1}(i)}\right)$};
				\node[transform shape] at (0,2*\ystep) (OBTC){\scriptsize$\scr{O}(n)\times \prod_{i\in\underline{n}}\left(\scr{O}(g^{-1}(i))\times \C^{g^{-1}(i)}\right)$};
				\node[transform shape] at (\xstep,\ystep)(PBTD) {\scriptsize$\scr{P}(n)\times \prod_{i\in\underline{n}}\left(\scr{P}(g^{-1}(i))\times \D^{f^{-1}(i)}\right)$};
				\node[transform shape] at (\xstep,2*\ystep)(OBTD) {\scriptsize$\scr{O}(n)\times \prod_{i\in\underline{n}}\left(\scr{O}(g^{-1}(i))\times \D^{g^{-1}(i)}\right)$};
				
				\node[transform shape] at (0,0) (PCm) {\scriptsize$\P(m)\times \C^{\underline{m}}$};
				\node[transform shape] at (\xstep,0)(PDm) {\scriptsize$\P(m)\times \D^{\underline{m}}$};
				
				\draw[->,transform shape] (PBTC) -- node[above]{\tiny$\id\times \Pi_i(\id\times F^{g^{-1}(i)})$} (PBTD);
				\draw[->,transform shape] (OBTC) -- node[above]{\tiny$\id\times \Pi_i(\id\times F^{g^{-1}(i)})$} (OBTD);
				\draw[->,transform shape] (OBTC) -- (PBTC);
				\draw[->,transform shape] (OBTD) -- (PBTD);
				\draw[->,transform shape] (PBTC) -- (PCm);
				\draw[->,transform shape] (PBTD) -- (PDm);
				\draw[->,transform shape] (PCm) -- (PDm);
			\end{scope}
			
			\begin{scope}[canvas is zy plane at x=3]
				\node[transform shape] at (0,\ystep) (PCn){\scriptsize$\scr{P}(n)\times \C^n$};
				\node[transform shape] at (0,2*\ystep) (OCn){\scriptsize$\scr{O}(n)\times \C^n$};
				\node[transform shape] at (\xstep,\ystep)(PDn) {\scriptsize$\scr{P}(n)\times \D^n$};
				\node[transform shape] at (\xstep,2*\ystep)(ODn) {\scriptsize$\scr{O}(n)\times \D^n$};
				
				\node[transform shape] at (0,0) (C) {\scriptsize$\C$};
				\node[transform shape] at (\xstep,0)(D) {\scriptsize$\D$};
				
				\draw[->,transform shape] (PCn) -- node[above]{\tiny$a$} (PDn);
				\draw[->,transform shape] (OCn) -- node[above]{\tiny$a$} (ODn);
				\draw[->,transform shape] (OCn) -- (PCn);
				\draw[->,transform shape] (PCn) -- (C);
				\draw[->,transform shape] (ODn) -- (PDn);
				\draw[->,transform shape] (PDn) -- (D);
				\draw[->,transform shape] (C) -- (D);
				\draw[-{Implies},double distance=1.5pt,transform shape] (C) to node[below right] {$\xi_n$} (PDn);
			\end{scope}
			
			\begin{scope}[canvas is zy plane at x=-6]
				\node[transform shape] at (0,\ystep) (OCm){\scriptsize$\scr{O}(m)\times \C^m$};
				\node[transform shape] at (\xstep,\ystep)(ODm) {\scriptsize$\scr{O}(m)\times \D^m$};
				
				\draw[->,transform shape] (OCm) -- node[below]{\tiny$a$} (ODm);
			\end{scope}
			
			\begin{scope}[canvas is xy plane at z=0]
				\draw[-{Implies},double distance=1.5pt,transform shape] (PCm) to node[above] {$\phi_g$} (PCn);
				\draw[-{Implies},double distance=1.5pt,transform shape] (PDm) to node[above] {$\psi_g$} (PDn);
			\end{scope}
			\begin{scope}[canvas is xz plane at y=5]
				\draw[-{Implies},double distance=1.5pt,transform shape] (PCn) to node[right] {\tiny$\id\times \Pi_i\xi_{g^{-1}(i)}$} (PBTD);
				\draw[-{Implies},double distance=1.5pt,transform shape] (C) to node[left] {$\xi_m$} (PDm);
				\draw[-{Implies},double distance=1.5pt,transform shape] (OCn) to node[right] {\tiny$\id\times \Pi_i(f^\ast\xi)_{g^{-1}(i)}$} (OBTD);
			\end{scope}
			
			\draw[->] (PBTC) -- (PCn);
			\draw[->] (PBTD) -- (PDn);
			\draw[->] (OBTC) -- (OCn); 
			\draw[->] (OBTD) -- (ODn);
			\draw[->] (PCm) -- (C);
			\draw[->] (PDm) -- (D);
			
			\draw[->] (OBTC) -- (OCm);
			\draw[->] (OBTD) -- (ODm);
			\draw[->] (OCm) -- (PCm);
			\draw[->] (ODm) -- (PDm);
			\path (15,7) node{};
		\end{tikzpicture}
	\end{equation}
	
\end{document}